\documentclass{article}
\usepackage{hyperref}
\usepackage{amsmath}
\usepackage{amsfonts}
\usepackage{amssymb}
\usepackage{amsthm}
\usepackage{graphicx}
\graphicspath{{./Pictures/}}
\newtheorem{theorem}{Theorem}

\newtheorem{lemma}[theorem]{Lemma}

\newtheorem{remark}[theorem]{Remark}

\date{}
\begin{document}

\title{Benchmark Computation of Eigenvalues with Large Defect for Non-Selfadjoint
Elliptic Differential Operators \thanks{ The work of the second author has
been funded by the Austrian Science Fund (FWF) through the project P
29197-N32.}}
\author{Rebekka Gasser\thanks{Institut f\"{u}r Mathematik, Universit\"{a}t Z\"{u}rich,
CH-8057 Z\"{u}rich, Switzerland, \texttt{rebekka.klara@hispeed.ch}. This work
is partially based on the master's thesis \cite{GasserMaster} of this author.}
\and Joscha Gedicke\thanks{University of Vienna, Faculty of Mathematics,
Oskar-Morgenstern-Platz 1, 1090 Vienna, Austria,
\texttt{joscha.gedicke@univie.ac.at}}
\and Stefan Sauter\thanks{Institut f\"ur Mathematik, Universit\"at Z\"urich,
CH-8057 Z\"urich, Switzerland, \texttt{stas@math.uzh.ch}} }
\maketitle

\begin{abstract}
In this paper we present benchmark problems for non-selfadjoint elliptic
eigenvalue problems with large defect and ascent. We describe the derivation
of the benchmark problem with a discontinuous coefficient and mixed boundary
conditions. Numerical experiments are performed to investigate the convergence
of a Galerkin finite element method with respect to the discretization
parameters, the regularity of the problem, and the ascent of the eigenvalue.
This allows us to verify the sharpness of the theoretical estimates from the
literature with respect to these parameters. We provide numerical evidence
about the size of the ascent and show that it is important to consider the
mean value for the eigenvalue approximation.

\end{abstract}

\section{Introduction}

The spectral theory and spectral analysis for elliptic operators have numerous
important practical applications in science and engineering and there are also
many mathematical applications. If the operator is non-selfadjoint
and/or has complex-valued coefficients in the operator and/or boundary
conditions, the arising sesquilinear form in the variational formulation is
not hermitian. Such problems arise frequently, e.g., in electromagnetic
scattering in lossy media, or if impedance/Sommerfeld-type boundary conditions
are imposed (see, e.g., \cite{Jackson02engl}). Also for the problem of
modeling \textit{mechanical} vibrations, non-selfadjoint eigenproblems arise
in many applications -- here, defective eigenvalues can be interpreted
physically as the transition point between an oscillatory and a monotonically
decaying behavior (see, e.g., \cite{DemmelNonsymEig}). As a consequence the
algebraic multiplicity of an eigenvalue can differ from the geometric
multiplicity and this has strong consequences for their numerical
approximation. Classical textbooks on this topic include \cite{Chatelinbook},
\cite{Dautray_Lions_III}, \cite{DunfordSchwartzII}, \cite{Kato} in the
mathematical and \cite{ewins1984modal}, \cite{fu2001modal},
\cite{natke2013einfuhrung} in the engineering literature. While the numerical
a priori/a posteriori analysis and the numerical simulation of eigenvalue
problems for \textit{selfadjoint }problems are fairly matured and numerous
monographs and textbook chapters exist in the mathematical and engineering
literature \cite{BO}, \cite{Chatelin1974}, \cite{Chatelinbook},
\cite{dyakonov96}, \cite{Hackell}, \cite{HR}, \cite{knyazev86},
\cite{knyazevosborn}, \cite{ovtchinnikov06}, \cite{Sauter_EV_publ},
\cite{StrangFix} the numerical computation of non-selfadjoint eigenvalue
problems is less developed. The standard reference for the numerical analysis
of the Galerkin finite element discretization is the seminal book chapter by
Babu\v{s}ka and Osborn \cite{BO}; see also \cite{CGMM}, \cite{Chatelin1974},
\cite{Chatelinbook}, \cite{GC}, \cite{GGMO}, \cite{Hackell}, \cite{HR},
\cite{YSBL}. They derive estimates for the convergence rates depending on the
mesh size, the polynomial order of the finite element space, the regularity of
the elliptic operator, and also on the \textit{ascent} of \textit{defective} eigenvalues.

To the best of our knowledge, systematic numerical experiments on the
sharpness of these estimates with respect to all parameters do not exist in
the literature and it is the goal of our paper to derive benchmark problems
for elliptic eigenvalue problems with possible large defects and ascents and
to verify by numerical experiments the sharpness of the estimates in \cite{BO}.

The construction of elliptic eigenvalue problems with large defect is far from
being trivial and very sensitive with respect to the choice of parameters in
the elliptic operator and boundary condition. We have generalized the
one-dimensional Green's function approach in \cite{F} and \cite{M} in order to
construct eigenvalue problems with large ascent and defect also in higher
dimension. The numerical experiments show very nicely that the estimates in
\cite{BO} are sharp with respect to all parameters.

The paper is structured as follows.

In Section \ref{ChapContProblem} we present the elliptic eigenvalue problem
with appropriate coefficients and transform it to an equivalent eigenvalue
problem for a compact operator. The Galerkin finite element discretization is
introduced in Section \ref{SecFED} and we recall briefly the estimates for the
convergence rates of the eigenvalues and eigenfunctions from \cite{BO}.
Section \ref{SecGreensFct} is devoted to the construction of elliptic
eigenvalue problems with large defect and ascent. We generalize the
one-dimensional Green's function approach from \cite{F} and \cite{M} to higher
dimensions and to eigenvalues with larger defect and ascent. In Section
\ref{SecNumExp} we present the results of numerical experiments and compare
them to the theoretical predictions. These examples show the sharpness of the
estimates in \cite{BO}.

\section{Elliptic Eigenvalue Problems\label{ChapContProblem}}

The computation of eigensystems of partial differential operators is of utmost
practical and mathematical importance and their efficient numerical
computation is one major field in numerical analysis and scientific computing.
Compared to selfadjoint eigenvalue problems for positive definite operators,
numerical methods for the solution of non-selfadjoint eigenvalue problems are
less developed, in particular, for problems with \textit{defective
}eigenvalues, i.e., eigenvalues where the algebraic and geometric multiplicity
of an eigenvalue are different.

As our model eigenvalue problem we consider the elliptic problem:%
\begin{align*}
-\operatorname{div}(a\nabla u)  &  =\lambda u\quad\text{in }\Omega,\\
u  &  =0\quad\text{on }\Gamma_{\operatorname*{D}},\\
a\nabla u\cdot\mathbf{n}+cu  &  =0\quad\text{on }\Gamma_{\operatorname*{R}},
\end{align*}
such that the arising variational formulation, in general, is
\textit{non-selfadjoint}. Here $\Omega\subset\mathbb{R}^{d}$ is a bounded
Lipschitz domain whose boundary $\Gamma$ is split into two disjoint measurable
subsets; the Dirichlet part $\Gamma_{\operatorname*{D}}$ and the Robin part
$\Gamma_{\operatorname*{R}}$. We always assume that $\Gamma_{\operatorname*{R}%
}$ has positive surface measure. The unit normal vector field $\mathbf{n}%
:\Gamma\rightarrow\mathbb{S}_{d-1}$ is defined almost everywhere and oriented
towards the exterior of $\Omega$. Let $L^{2}\left(  \Omega\right)  $ denote
the usual Lebesgue space with (complex) scalar product $\left(  u,v\right)
=\int_{\Omega}u\overline{v}$ and norm $\left\Vert \cdot\right\Vert =\left(
\cdot,\cdot\right)  ^{1/2}$. Let $H^{1}\left(  \Omega\right)  $ denote the
standard Sobolev space. We set $V:=H^{1}\left(  \Omega\right)  $ if
$\Gamma_{\operatorname*{D}}=\emptyset$ and $V:=\left\{  u\in H^{1}\left(
\Omega\right)  \mid\gamma_{\operatorname*{D}}u=0\right\}  $ in case that
$\Gamma_{\operatorname*{D}}$ has positive boundary measure. The standard trace
operators are denoted by $\gamma_{\operatorname*{D}}:H^{1}\left(
\Omega\right)  \rightarrow\Gamma_{\operatorname*{D}}$ and $\gamma
_{\operatorname*{R}}:H^{1}\left(  \Omega\right)  \rightarrow\Gamma
_{\operatorname*{R}}$. If the $\left(  d-1\right)  $-dimensional surface
measure $\left\vert \Gamma_{\operatorname*{R}}\right\vert $ is positive, the
multiplicative trace inequality holds%
\begin{equation}
\left\Vert \gamma_{\operatorname*{R}}u\right\Vert _{L^{2}\left(
\Gamma_{\operatorname*{R}}\right)  }^{2}\leq C_{\operatorname*{trace}%
}\left\Vert u\right\Vert \left\Vert u\right\Vert _{H^{1}\left(  \Omega\right)
}. \label{mti}%
\end{equation}
(For $d=2,3$, this is the last formula in \cite[p.41]{Grisvard85}. For $d=1$
it can be obtained by applying the fundamental theorem of calculus to the functions
$Z\left\vert u\right\vert ^{2}$ for a suitable chosen affine function $Z$).

The variational formulation of the eigenvalue problem is given by: Find
$\left(  u,\lambda\right)  \in V\backslash\left\{  0\right\}  \times
\mathbb{C}$ such that%
\begin{equation}
A\left(  u,v\right)  =\lambda\left(  u,v\right)  \quad\forall v\in V
\label{eigenvalueproblem}%
\end{equation}
with%
\[
A\left(  u,v\right)  :=\left(  a\nabla u,\nabla v\right)  +\left(
cu,v\right)  _{L^{2}\left(  \Gamma_{\operatorname*{R}}\right)  }.
\]
We assume%
\begin{equation}
a\in L^{\infty}\left(  \Omega\right)  \quad\text{and\quad}\underset{x\in
\Omega}{\operatorname*{ess}\inf}\left(  \operatorname{Re}a\left(  x\right)
\right)  \geq\alpha_{0}\quad\text{for some }\alpha_{0}>0, \label{realow}%
\end{equation}
and that $a$ is sufficiently smooth in an $\Omega$-neighborhood of
$\Gamma_{\operatorname*{R}}$ such that the trace $\gamma_{\operatorname*{R}%
}\left(  a\right)  $ is well-defined. Finally we assume that%
\begin{equation}
c\in L^{\infty}\left(  \Gamma_{\operatorname*{R}}\right)  \quad\text{and
set\quad}c_{0}:=\underset{x\in\Gamma_{\operatorname*{R}}}{\operatorname*{ess}%
\inf}\left(  \operatorname{Re}c\left(  x\right)  \right)  . \label{condc}%
\end{equation}
Clearly the sesquilinear form $A$ is non-selfadjoint if $\operatorname{Im}%
c\neq0$ on $\Gamma_{\operatorname*{R}}$ or $\operatorname{Im}a\neq0$ in the
$L^{\infty}$ sense.

\begin{lemma}
\label{LaxMilgram}Suppose (\ref{realow}) and (\ref{condc}) are satisfied. Let%
\[
\Lambda:=\left\{
\begin{array}
[c]{ll}%
0 & \text{if }c_{0}\geq0\wedge\left\vert \Gamma_{\operatorname*{D}}\right\vert
>0,\\
\alpha_{0} & \text{if }c_{0}\geq0\wedge\left\vert \Gamma_{\operatorname*{D}%
}\right\vert =0,\\
\alpha_{0}+\frac{C_{\operatorname*{trace}}\left\vert c_{0}\right\vert
^{2}}{\alpha_{0}} & \text{if }c_{0}<0.
\end{array}
\right.
\]
Then there exist constants $C_{\operatorname*{cont}}$, $c_{\operatorname*{coer}%
}>0$ such that the modified sesquilinear form $A_{\Lambda}\left(  u,v\right)
:=A\left(  u,v\right)  +\Lambda\left(  u,v\right)  $ satisfies%
\begin{align}
\left\vert A_{\Lambda}\left(  u,v\right)  \right\vert  &  \leq
C_{\operatorname*{cont}}\left\Vert u\right\Vert _{H^{1}\left(  \Omega\right)
}\left\Vert v\right\Vert _{H^{1}\left(  \Omega\right)  }\quad\forall u,v\in
V,\nonumber\\
\operatorname{Re}A_{\Lambda}\left(  u,u\right)   &  \geq
c_{\operatorname*{coer}}\left\Vert u\right\Vert _{H^{1}\left(  \Omega\right)
}^{2}\quad\forall u\in V. \label{coercivity}%
\end{align}

\end{lemma}

\begin{proof}
\textbf{a) Continuity. }Let $\alpha_{1}:=\left\Vert a\right\Vert _{L^{\infty
}\left(  \Omega\right)  }$ and $c_{1}:=\left\Vert c\right\Vert _{L^{\infty
}\left(  \Gamma_{\operatorname*{R}}\right)  }$. For $u,v\in H^{1}\left(
\Omega\right)  $ it holds that
\begin{align*}
\left\vert A_{\Lambda}\left(  u,v\right)  \right\vert  &  \leq\alpha
_{1}\left\Vert \nabla u\right\Vert \left\Vert \nabla v\right\Vert
+\Lambda\left\Vert u\right\Vert \left\Vert v\right\Vert +c_{1}\left\Vert
u\right\Vert _{L^{2}\left(  \Gamma_{\operatorname*{R}}\right)  }\left\Vert
v\right\Vert _{L^{2}\left(  \Gamma_{\operatorname*{R}}\right)  }\\
&  \leq\alpha_{1}\left\Vert \nabla u\right\Vert \left\Vert \nabla v\right\Vert
+\Lambda\left\Vert u\right\Vert \left\Vert v\right\Vert +c_{1}%
C_{\operatorname*{trace}}\left\Vert u\right\Vert _{H^{1}\left(
\Omega\right)  }\left\Vert v\right\Vert _{H^{1}\left(  \Omega\right)  }\\
&  \leq\left(  \alpha_{1}+\Lambda+c_{1}C_{\operatorname*{trace}}^{2}\right)
\left\Vert u\right\Vert _{H^{1}\left(  \Omega\right)  }\left\Vert v\right\Vert
_{H^{1}\left(  \Omega\right)  }.
\end{align*}

\textbf{b) Coercivity.} To prove coercivity of $A_{\Lambda}$ we begin with%
\[
\operatorname{Re}A_{\Lambda}\left(  u,u\right)  \geq\alpha_{0}\left\Vert
\nabla u\right\Vert ^{2}+\Lambda\left\Vert u\right\Vert ^{2}+c_{0}\left\Vert
u\right\Vert _{L^{2}\left( \Gamma_{\operatorname*{R}}\right) }^{2}.
\]
If $\left\vert \Gamma_{\operatorname*{D}}\right\vert >0$ and $c_{0}\geq0$, we
have $\Lambda=0$ and the Friedrichs inequality (with constant
$c_{\operatorname{F}}>0$) implies%
\[
\operatorname{Re}A_{\Lambda}\left(  u,u\right)  \geq c_{\operatorname{F}%
}\alpha_{0}\left\Vert u\right\Vert _{H^{1}\left(  \Omega\right)  }^{2}%
\quad\forall u\in V.
\]
If $c_{0}\geq0$ and $\left\vert \Gamma_{\operatorname*{D}}\right\vert =0$, the
choice of $\Lambda$ leads to%
\[
\operatorname{Re}A_{\Lambda}\left(  u,u\right)  \geq\alpha_{0}\left\Vert
u\right\Vert _{H^{1}\left(  \Omega\right)  }^{2}\quad\forall u\in V.
\]
If $c_{0}<0$, we employ the multiplicative trace inequality (\ref{mti}) and a
Young's inequality for $\varepsilon>0$%
\begin{align*}
\operatorname{Re}A_{\Lambda}\left(  u,u\right)   &  \geq\alpha_{0}\left\Vert
\nabla u\right\Vert ^{2}+\Lambda\left\Vert u\right\Vert ^{2}-\left\vert
c_{0}\right\vert \left\Vert u\right\Vert _{L^{2}\left( \Gamma_{\operatorname*{R}}\right) }^{2}\\
&  \geq\alpha_{0}\left\Vert \nabla u\right\Vert ^{2}+\Lambda\left\Vert
u\right\Vert ^{2}-C_{\operatorname*{trace}}\left\vert c_{0}\right\vert
\left\Vert u\right\Vert \left\Vert u\right\Vert _{H^{1}\left(  \Omega\right)
}\\
&  \geq\alpha_{0}\left\Vert \nabla u\right\Vert ^{2}+\Lambda\left\Vert
u\right\Vert ^{2}-C_{\operatorname*{trace}}\left\vert c_{0}\right\vert \left(
\frac{\varepsilon}{2}\left\Vert \nabla u\right\Vert ^{2}+\left(
\frac{\varepsilon}{2}+\frac{1}{2\varepsilon}\right)  \left\Vert u\right\Vert
^{2}\right)  .
\end{align*}
The choice $\varepsilon=\frac{\alpha_{0}}{2C_{\operatorname*{trace}}\left\vert
c_{0}\right\vert }$ leads to%
\[
\operatorname{Re}A_{\Lambda}\left(  u,u\right)  \geq\frac{3\alpha_{0}}%
{4}\left\Vert \nabla u\right\Vert ^{2}+\left(  \Lambda-\left(  \frac
{\alpha_{0}}{4}+\frac{C_{\operatorname*{trace}}^{2}\left\vert c_{0}\right\vert
^{2}}{\alpha_{0}}\right)  \right)  \left\Vert u\right\Vert ^{2}\geq
\frac{3\alpha_{0}}{4}\left\Vert u\right\Vert _{H^{1}\left(  \Omega\right)
}^{2}.
\]
\end{proof}

Lemma \ref{LaxMilgram} implies via the Lax-Milgram lemma that for any
continuous anti-linear functional $f\in V^{\times}$, where $V^{\times}$ denotes the dual of $V$, the problem:%
\begin{equation}
\text{find }u\in V\text{ such that }A_{\Lambda}\left(  u,v\right)  =\left(
f,v\right)  \quad\forall v\in V\label{opeq}%
\end{equation}
has a unique solution. Throughout the paper we identify the $L^{2}\left(
\Omega\right)  $ scalar product with its continuous extension to the
anti-linear pairing on $V^{\times}\times V$.

We say that the problem has regularity $r>0$ if for any $0\leq s\leq r$, there
exists a constant $C_{s}$ such that for any $f\in H^{s-1}\left(
\Omega\right)  $ the solution of (\ref{opeq}) is in $H^{1+s}\left(
\Omega\right)  \cap V$ and satisfies%
\begin{equation}
\left\Vert u\right\Vert _{H^{1+s}\left(  \Omega\right)  }\leq C_{s}\left\Vert
f\right\Vert _{H^{s-1}\left(  \Omega\right)  }. \label{regprimal}%
\end{equation}
We say it has adjoint regularity $r_{\ast}>0$ if for any $0\leq s\leq r_{\ast
}$, there exists a constant $C_{s}^{\ast}$ such that for any $g\in
H^{s-1}\left(  \Omega\right)  $ the solution of the adjoint problem%
\begin{equation}
\text{find }z\in V\text{ such that }A_{\Lambda}\left(  v,z\right)  =\left(
v,g\right)  \quad\forall v\in V
\end{equation}
is in $H^{1+s}\left(  \Omega\right)  \cap V$ and satisfies%
\begin{equation}
\left\Vert z\right\Vert _{H^{1+s}\left(  \Omega\right)  }\leq C_{s}^{\ast
}\left\Vert g\right\Vert _{H^{s-1}\left(  \Omega\right)  }. \label{regadjoint}%
\end{equation}
It is well known that
\[
\min\left\{  r,r_{\ast}\right\}  \geq r_{0}>0
\]
for some $r_{0}$ depending on the geometry of the domain, the geometry of the
discontinuities in the coefficient $a$, as well as on $\alpha_{0}$ and
$\left\Vert a\right\Vert _{L^{\infty}\left(  \Omega\right)  }$.

From the compact embedding $V \overset
{\operatorname*{c}}{\hookrightarrow}L^{2}\left(  \Omega\right)$ and Lemma \ref{LaxMilgram} it follows
that there exists a compact operator $K_{\Lambda}:V\rightarrow V$ such that
\[
A_{\Lambda}\left(  K_{\Lambda}u,v\right)  =\left(  u,v\right)  \qquad\forall
u,v\in V.
\]
From the theory of compact operators we deduce that (\ref{eigenvalueproblem})
is equivalent to the eigenvalue problem: Find $\left(  u,\mu\right)  \in
V\backslash\left\{  0\right\}  \times\mathbb{C}$ such that%
\begin{equation}
K_{\Lambda}u=\mu u. \label{opform}%
\end{equation}
The eigenfunctions are the same as for the original problem
(\ref{eigenvalueproblem}) and the eigenvalues are related by%
\[
\frac{1}{\mu}=\lambda+\Lambda.
\]
This allows to apply the spectral theory for compact operators to our problem:
From (\ref{coercivity}) we conclude that $\operatorname{Re}\mu>0$. The
smallest integer $\alpha$ such that $\mathcal{N}\left(  \left(  K_{\Lambda
}-\mu I\right)  ^{\alpha}\right)  =\mathcal{N}\left(  \left(  K_{\Lambda}-\mu
I\right)  ^{\alpha+1}\right)  $ (where $\mathcal{N}$ denotes the null space)
is called \textit{ascent of} $K_{\Lambda}-\mu I$ and is finite for compact
operators. The integer $m_{\operatorname*{alg}}=\dim\mathcal{N}\left(  \left(
K_{\Lambda}-\mu I\right)  ^{\alpha}\right)  $ is the \textit{algebraic
multiplicity} of $\mu$ and is finite. The subspace $\mathcal{N}\left(  \left(
K_{\Lambda}-\mu I\right)  ^{\alpha}\right)  $ is called the space of
\textit{generalized eigenfunctions} corresponding to the eigenvalue $\mu$. The
\textit{geometric multiplicity} is equal to $m_{\operatorname{geo}}%
:=\dim\mathcal{N}\left(  K_{\Lambda}-\mu I\right)  $ and is always less than
or equal to $m_{\operatorname*{alg}}$. If $m_{\operatorname{geo}%
}<m_{\operatorname*{alg}}$ we say that the eigenvalue $\mu$ is
\textit{defective}.

\section{Finite Element Discretization\label{SecFED}}

Let $\mathcal{T}=\left\{  K_{i},1\leq i\leq N\right\}  $ denote a conforming
finite element mesh for the domain $\Omega$ (see, e.g., \cite{scottbrenner3},
\cite{CiarletPb}) consisting of (closed) simplices $K$. Let $h_{K}%
:=\operatorname*{diam}K$ and $h:=\max\left\{  h_{K}:K\in\mathcal{T}\right\}  $
and let $\rho_{K}$ denote the diameter of the largest inscribed ball in $K$.
We assume that the mesh is shape regular, i.e., all constants in the error
estimates, in general, depend continuously on the shape-regularity constant%
\[
c_{\operatorname*{sr}}:=\max\left\{  \frac{h_{K}}{\rho_{K}}:K\in
\mathcal{T}\right\}
\]
and, possibly, increase for large $c_{\operatorname*{sr}}$. The finite element
space is defined by%
\[
V_{h}:=\left\{  u\in V\mid\forall K\in\mathcal{T}:\left.  u\right\vert _{K}%
\in\mathbb{P}_{p}\right\}  ,
\]
where $\mathbb{P}_{p}$ denotes the space of $d-$variate polynomials of total
degree $p$.

The Galerkin finite element method to discretize the eigenvalue problem is
given by:
\begin{equation}
\text{find }\left(  u_{j,h},\lambda_{j,h}\right)  \in V_{h}\backslash\left\{
0\right\}  \times\mathbb{C}\text{\quad such that\quad}A\left(  u_{j,h}%
,v\right)  =\lambda_{j,h}\left(  u_{j,h},v\right)  \quad\forall v\in
V_{h}.\label{GalForm}%
\end{equation}
As in the continuous setting (\ref{opform}), this problem can be reformulated
as an operator equation. Let $K_{\Lambda,h}:V_{h}\rightarrow V_{h}$ be given
by%
\[
A_{\Lambda}\left(  K_{\Lambda,h}u,v\right)  =\left(  u,v\right)  \qquad\forall
u,v\in V_{h}.
\]
Then (\ref{GalForm}) is equivalent to: Find $\left(  u_{j,h},\mu_{j,h}\right)
\in V_{h}\backslash\left\{  0\right\}  \times\mathbb{C}$ such that%
\[
K_{\Lambda,h}u_{j,h}=\mu_{j,h}u_{j,h}%
\]
and the relation $1/\mu_{j,h}=\lambda_{j,h}+\Lambda$ holds.

In the seminal work by Babu\v{s}ka and Osborn \cite{BO} the theory for the
numerical solution of eigenvalue problems for elliptic, possibly
non-selfadjoint differential operators has been developed. 
One important
result is that, for a defective eigenvalue, the convergence rate suffers from
an ascent which is larger than one. From the a priori error analysis of the
finite element method \cite[Theorem 8.3]{BO}, we have that
\[
|\lambda_j-\lambda_{j,h}|\leq Ch^{\min(r+r^{\ast},2p)/\alpha},
\]
for the regularity $\min\left\{  r,r^{\ast}\right\}  >0$ (cf. (\ref{regprimal}%
), (\ref{regadjoint})) of the original and adjoint sesquilinear form, the
polynomial degree $p>0$ and the ascent $\alpha$ of the eigenvalue $\lambda$.
Note that the convergence of single eigenvalues deteriorates for large
$\alpha$. In contrast \cite[Theorem 8.2]{BO} states that 
there are eigenvalues $\lambda_{j_s,h}$, $1\leq s \leq m_{\operatorname{alg}}$,
that converge towards $\lambda_j$, and
the convergence rate
of the mean eigenvalue is independent of the defect%
\[
\left\vert \lambda_j-\left(  \frac{1}{m_{\operatorname*{alg}}}\sum
_{s=1}^{m_{\operatorname*{alg}}}\lambda_{j_s,h}^{-1}\right)  ^{-1}\right\vert
\leq Ch^{\min(r+r^{\ast},2p)}.
\]
Moreover, for the convergence of the corresponding
eigenfunctions, we have the following result \cite[Theorem 8.4]{BO}. Suppose
that the discrete generalized eigenfunction $u_{j,h}$ satisfies $\left(
K_{\Lambda,h}-\mu_{j,h}I\right)  ^{k}u_{j,h}=0$, for some $0<k\leq\alpha$.
Then there exists for any $k\leq\ell\leq\alpha$, a generalized eigenfunction
$u_{j}$ in the continuous eigenspace such that $\left(  K_{\Lambda}-\mu
_{j}I\right)  ^{\ell}u_{j}=0$, and
\[
\left\Vert u_{j}-u_{j,h}\right\Vert _{H^{1}\left(  \Omega\right)  }\leq
Ch^{\min(r,p)(\ell-k+1)/\alpha}.
\]
In particular we have for $\ell=k$ that
\[
\left\Vert u_{j}-u_{j,h}\right\Vert _{H^{1}\left(  \Omega\right)  }\leq
Ch^{\min(r,p)/\alpha},
\]
or for $\ell=\alpha$ that
\[
\left\Vert u_{j}-u_{j,h}\right\Vert _{H^{1}\left(  \Omega\right)  }\leq
Ch^{\min(r,p)(\alpha-k+1)/\alpha}.
\]
Suppose that $k=1$, i.e. $u_{j,h}$ is an eigenfunction, then for $\ell=\alpha$
we get the expected convergence rate similar to simple eigenvalues
\[
\left\Vert u_{j}-u_{j,h}\right\Vert _{H^{1}\left(  \Omega\right)  }\leq
Ch^{\min(r,p)}.
\]
Since the rate of convergence of the eigenvalue error is usually related to
the rate of convergence of the associated eigenfunction, the question arises
if any rate between those two extreme cases can be observed in practice. This
motivates the construction of benchmark examples with (at least) one defective eigenvalue.

\section{One-dimensional Benchmark Problems\label{SecGreensFct}}

In this section, we employ the general Green's function approach for the construction
of defective eigenvalues. In \cite{F}, \cite{M} this approach has been used to
set up a one-dimensional boundary value problem with eigenvalues of defect 2. Here we
consider a more general one-dimensional boundary value problem to construct
eigenvalues with defect 3. Later, in Section~\ref{SecNumExp}, this will be generalized
to higher dimensional problems with even larger eigenvalue defects.

Let $\Omega=\left(  0,1\right)  $ be split into subdomains $\Omega
_{1}:=\left(  0,b\right)  $ and $\Omega_{2}:=\left(  b,1\right)  $ for
$b\in\left(  0,1\right)  $. Let $a:=\left\{
\begin{array}
[c]{ll}%
1 & \text{in }\Omega_{1}\\
a_{\operatorname*{R}} & \text{in }\Omega_{2}%
\end{array}
\right.  $ for $a_{\operatorname*{R}}\in\mathbb{C}$ with $\operatorname{Re}%
a_{\operatorname*{R}}>0$. 
We consider the following transmission problem
\begin{equation}%
\begin{array}
[c]{rll}%
L_{a}u & =\lambda u\text{,} & \text{in }\Omega\backslash\left\{  b\right\}
,\\
u\left(  0\right)  =0\text{,} & a_{R}u^{\prime}\left(  1\right)  +cu\left(
1\right)  =0\text{,} & \\
\left[  u\right]  _{b} & =\left[  au^{\prime}\right]  _{b}=0 &
\end{array}
\label{defl}%
\end{equation}
with $L_{a}u:=-\left(  au^{\prime}\right)  ^{\prime}$ and the jump $\left[
\cdot\right]  _{b}$ across $b$.

\begin{remark}
We do not discuss the case that $\lambda=0$ is an eigenvalue. This case can be
treated by the following analysis by adding $\Lambda u$ on both sides of the
first equation of (\ref{defl}) for some $\Lambda>0$ so that $\left(
L_{a}+\Lambda\right)  u=\tilde{\lambda}u$ with $\tilde{\lambda}=\lambda
+\Lambda>0$.
\end{remark}

Let $V=\left\{  v\in H^{1}\left(  \Omega\right)  \mid v\left(  0\right)
=0\right\}  $ and let $V^{\times}$ denote the space of anti-linear functionals
on $V$. Define the sesquilinear form $A:V\times V\rightarrow\mathbb{C}$ by
\[
A\left(  u,v\right)  :=\left(  au^{\prime},v^{\prime}\right)  +cu\left(
1\right)  \bar{v}\left(  1\right)  \qquad\forall v\in V.
\]
The continuity of $A$ (cf. Lem. \ref{LaxMilgram}) implies that there exists an
operator $\mathcal{A}:V\rightarrow V^{\times}$ such that%
\[
\left(  \mathcal{A}u,v\right)  =A\left(  u,v\right)  \qquad\forall u,v\in V.
\]

The weak form of (\ref{defl}) is given by: find $\left(  u,\lambda\right)  \in
V\backslash\left\{  0\right\}  \times\mathbb{C}$ such that%
\[
A\left(  u,v\right)  =\lambda\left(  u,v\right)  \qquad\forall v\in V
\]
or in operator form%
\[
\left(  \mathcal{A}-\lambda I\right)  u=0.
\]
In the following we will derive a representation of the exact solutions for
this problem which will allow us to determine choices of parameters $a_R,b,c$
such that an eigenvalue becomes defective. Let $\mu:=\left\{
\begin{array}
[c]{ll}%
\mu_{\operatorname*{L}}:=\sqrt{\lambda} & \text{in }\Omega_{1}\\
\mu_{\operatorname*{R}}:=\sqrt{\lambda/a_{\operatorname*{R}}} & \text{in
}\Omega_{2}%
\end{array}
\right.  $ and $v_{a,b,c}^{\operatorname*{L}}\left(  \lambda,x\right)
:=\sin\mu_{\operatorname*{L}}x$. We employ the ansatz%
\begin{equation}
u\left(  x\right)  =\left\{
\begin{array}
[c]{ll}%
A_{1}v_{a,b,c}^{\operatorname*{L}}\left(  \lambda,x\right)  & x\in\Omega
_{1},\\
A_{2}v_{a,b,c}^{\operatorname*{R}}\left(  \lambda,x\right)  & x\in\Omega_{2},
\end{array}
\right.  \label{repu}%
\end{equation}
where the coefficients $c_{1}$, $c_{2}$ in $v_{a,b,c}^{\operatorname*{R}%
}\left(  \lambda,x\right)  =c_{1}\sin\mu_{\operatorname*{R}}x+c_{2}\cos
\mu_{\operatorname*{R}}x$, $x\in\Omega_{2}$, are chosen such that
$a_{R}\partial_{x}v_{a,b,c}^{\operatorname*{R}}\left(  \lambda,1\right)
+cv_{a,b,c}^{\operatorname*{R}\left(  \lambda,1\right) } =0$ is
satisfied, i.e.,%
\begin{equation}
v_{a,b,c}^{\operatorname*{R}}\left(  \lambda,x\right)  =\left(  c\cos
\mu_{\operatorname*{R}}+\mu_{\operatorname*{R}}\sin\mu_{\operatorname*{R}%
}\right)  \sin\mu_{\operatorname*{R}}x+\left(  \mu_{\operatorname*{R}}\cos
\mu_{\operatorname*{R}}-c\sin\mu_{\operatorname*{R}}\right)  \cos
\mu_{\operatorname*{R}}x. \label{defvr}%
\end{equation}
We employ the transmission conditions $\left[  u\right]  _{b}=\left[
au^{\prime}\right]  _{b}=0$ to see that the coefficients $A_{1}$, $A_{2}$ in
(\ref{repu}) satisfy the linear relation%
\begin{equation}
\mathbf{M}_{a,b,c}\left(  \lambda\right)  \left(
\begin{array}
[c]{c}%
A_{1}\\
A_{2}%
\end{array}
\right)  =\mathbf{0\quad}\text{with\quad}\mathbf{M}_{a,b,c}\left(
\lambda\right)  :=\left[
\begin{array}
[c]{ll}%
v_{a,b,c}^{\operatorname*{L}}\left(  \lambda,b\right)  & -v_{a,b,c}%
^{\operatorname*{R}}\left(  \lambda,b\right) \\
\partial_{x}v_{a,b,c}^{\operatorname*{L}}\left(  \lambda,b\right)  &
-a\partial_{x}v_{a,b,c}^{\operatorname*{R}}\left(  \lambda,b\right)
\end{array}
\right]  . \label{MlinA1A2}%
\end{equation}
Hence, $\lambda\in\mathbb{C}$ is an eigenvalue of problem (\ref{defl}) if and
only if $\det\mathbf{M}_{a,b,c}\left(  \lambda\right)  =0$ since then,
(\ref{MlinA1A2}) has non-trivial solutions.

\begin{remark}
\label{RemZeroMat}For $\lambda\neq0$, the matrix $\mathbf{M}_{a,b,c}\left(
\lambda\right)  $ is not the zero matrix. Hence, the eigenspace of any
eigenvalue $\lambda_{j}\neq0$ has dimension $1$.
\end{remark}

\begin{lemma}
\label{Lemlambdaj}Let $\lambda_{j}$ be an eigenvalue of (\ref{defl}). Then,
there exists a neighborhood $\mathcal{U}\left(  \lambda_{j}\right)  $ such
that $v_{a,b,c}^{\operatorname*{R}}\left(  \lambda_{j},\cdot\right)  $ is not
the zero function for all $\lambda\in\mathcal{U}\left(  \lambda_{j}\right)  $,
i.e., the coefficients $c_{1}\left(  \lambda\right)  =c\cos\mu
_{\operatorname*{R}}+\mu_{\operatorname*{R}}\sin\mu_{\operatorname*{R}}$ and
$c_{2}\left(  \lambda\right)  =\mu_{\operatorname*{R}}\cos\mu
_{\operatorname*{R}}-c\sin\mu_{\operatorname*{R}}$ in (\ref{defvr}) are not
vanishing simultaneously in this neighborhood.
\end{lemma}

\begin{proof}
If the coefficients $a,b,c$ are such that $v_{a,b,c}^{\operatorname*{R}%
}\left(  \lambda_{j},\cdot\right)  $ is the zero function, then,
$v_{a,b,c}^{\operatorname*{L}}\left(  \lambda_{j},\cdot\right)  $ satisfies%
\begin{align*}
-\partial_{x}^{2}v_{a,b,c}^{\operatorname*{L}}\left(  \lambda_{j}%
,\cdot\right)   &  =\lambda_{j}v_{a,b,c}^{\operatorname*{L}}\left(
\lambda_{j},\cdot\right)  \quad\text{in }\Omega_{1},\\
v_{a,b,c}^{\operatorname*{L}}\left(  \lambda_{j},0\right)   &  =v_{a,b,c}%
^{\operatorname*{L}}\left(  \lambda_{j},b\right)  =\partial_{x}v_{a,b,c}%
^{\operatorname*{L}}\left(  \lambda_{j},b\right)  =0.
\end{align*}
However, this implies $v_{a,b,c}^{\operatorname*{L}}\left(  \lambda_{j}%
,\cdot\right)  =0$ and hence $\lambda_{j}$ cannot be an eigenvalue of
(\ref{defl}). By contradiction we may conclude that $v_{a,b,c}%
^{\operatorname*{R}}\left(  \lambda_{j},\cdot\right)  $ is not the zero
function. Its definition implies that the coefficients $c_{1}\left(
\lambda\right)  $, $c_{2}\left(  \lambda\right)  $ cannot vanish
simultaneously at $\lambda=\lambda_{j}$. Since $c_{1}$, $c_{2}$ depend
continuously on $\lambda$ this property carries over to a neighborhood
$\mathcal{U}\left(  \lambda_{j}\right)  $ of $\lambda_{j}$.
\end{proof}

In order to determine the defect and ascent of the eigenvalue and a basis for
the generalized eigenspace, we will employ the Green's function for problem
(\ref{MlinA1A2})%
\[%
\begin{tabular}
[c]{ll}%
\multicolumn{2}{l}{$L_{a,x}G_{a,b,c}\left(  \lambda,x,y\right)  -\lambda
G_{a,b,c}\left(  \lambda,x,y\right)  =\delta\left(  x-y\right)  ,\quad
\text{for }\left(  x,y\right)  \in\Omega\backslash\left\{  b\right\}
\times\Omega,$}\\
$G_{a,b,c}\left(  \lambda,0,y\right)  =0\quad\text{and}\quad a_{R}\partial
_{x}G_{a,b,c}\left(  \lambda,1,y\right)  +cG_{a,b,c}\left(  \lambda
,1,y\right)  =0,$ & $\text{in }\Omega,$\\
$\left[  G_{a,b,c}\left(  \lambda,\cdot,y\right)  \right]  _{b}=\left[
a\partial_{x}G_{a,b,c}\left(  \lambda,\cdot,y\right)  \right]  _{b}=0,$ &
$\text{in }\Omega.$%
\end{tabular}
\]
Here, the subscript $x$ in $L_{a,x}$ indicates that the differential operator
is applied with respect to the $x$ variable. It is an easy exercise to prove
that the Green's function is given by%
\[
G_{a,b,c}\left(  \lambda,x,y\right)  :=G_{a,b,c}^{\operatorname*{free}}\left(
\lambda,x,y\right)  +G_{a,b,c}^{\hom}\left(  \lambda,x,y\right)
\]
with%
\[
G_{a,b,c}^{\operatorname*{free}}\left(  \lambda,x,y\right)  :=-\frac
{\operatorname*{e}^{\operatorname*{i}\mu\left\vert x-y\right\vert }%
}{2\operatorname*{i}\mu}+\left\{
\begin{array}
[c]{ll}%
\frac{\operatorname*{e}^{\operatorname*{i}\mu_{\operatorname*{L}}y}%
}{2\operatorname*{i}\mu_{\operatorname*{L}}}\cos\mu_{\operatorname*{L}}x &
\text{in }\Omega_{1},\\
\frac{\operatorname*{e}^{-\operatorname*{i}\mu_{\operatorname*{R}}y}%
}{2\operatorname*{i}\mu_{\operatorname*{R}}}\operatorname*{e}%
\nolimits^{\operatorname*{i}\mu_{\operatorname*{R}}x} & \text{in }\Omega_{2}.
\end{array}
\right.
\]
and%
\[
G_{a,b,c}^{\hom}\left(  \lambda,x,y\right)  :=\left\{
\begin{array}
[c]{ll}%
G_{1}\left(  y\right)  v_{a,b,c}^{\operatorname*{L}}\left(  \lambda,x\right)
& \text{in }\Omega_{1}\times\Omega,\\
G_{2}\left(  y\right)  v_{a,b,c}^{\operatorname*{R}}\left(  \lambda,x\right)
& \text{in }\Omega_{2}\times\Omega.
\end{array}
\right.
\]
Note that the \textit{boundary conditions} are already incorporated into
$G_{a,b,c}^{\operatorname*{free}}$. The coefficient functions $G_{1}$, $G_{2}$
are the solution of the system of linear equations
\begin{equation}
\mathbf{M}_{a,b,c}\left(  \lambda\right)  \left(
\begin{array}
[c]{c}%
G_{1}\left(  y\right) \\
G_{2}\left(  y\right)
\end{array}
\right)  =\left(
\begin{array}
[c]{c}%
g_{a,b,c}\left(  \lambda,y\right) \\
f_{a,b,c}\left(  \lambda,y\right)
\end{array}
\right)  \text{;\ }\left\{
\begin{array}
[c]{l}%
g_{a,b,c}\left(  \lambda,y\right)  :=\left[  G_{a,b,c}^{\operatorname*{free}%
}\left(  \lambda,\cdot,y\right)  \right]  _{b},\\
f_{a,b,c}\left(  \lambda,y\right)  :=\left[  a\partial_{x}G_{a,b,c}%
^{\operatorname*{free}}\left(  \lambda,\cdot,y\right)  \right]  _{b}.
\end{array}
\right.  \label{defG1G2}%
\end{equation}

\begin{lemma}
\label{LemLinIndep}Let $\lambda\neq0$. Then, the functions $g_{a,b,c}\left(
\lambda,\cdot\right)  $ and $f_{a,b,c}\left(  \lambda,\cdot\right)  $ are
linearly independent.
\end{lemma}

\begin{proof}
We have%
\begin{align*}
g_{a,b,c}\left(  \lambda,y\right)   &  =\frac{\operatorname*{e}%
^{\operatorname*{i}\mu_{\operatorname*{R}}\left(  b-y\right)  }%
-\operatorname*{e}^{\operatorname*{i}\mu_{\operatorname*{R}}\left\vert
b-y\right\vert }}{2\operatorname*{i}\mu_{\operatorname*{R}}}+\frac
{\operatorname*{e}^{\operatorname*{i}\mu_{\operatorname*{L}}\left\vert
b-y\right\vert }-\operatorname*{e}^{\operatorname*{i}\mu_{\operatorname*{L}}%
y}\cos\mu_{\operatorname*{L}}b}{2\operatorname*{i}\mu_{\operatorname*{L}}},\\
f_{a,b,c}\left(  \lambda,y\right)   &  =\frac{\operatorname*{sign}\left(
b-y\right)  }{2}\delta\left(  y\right)  +a_{\operatorname*{R}}\frac
{\operatorname*{e}^{\operatorname*{i}\mu_{\operatorname*{R}}\left(
b-y\right)  }}{2}+\frac{\operatorname*{e}^{\operatorname*{i}\mu
_{\operatorname*{L}}y}}{2\operatorname*{i}}\sin\mu_{\operatorname*{L}}b
\end{align*}
for $\delta\left(  y\right)  :=\operatorname*{e}\nolimits^{\operatorname*{i}%
\mu_{\operatorname*{L}}\left\vert b-y\right\vert }-a_{\operatorname*{R}%
}\operatorname*{e}\nolimits^{\operatorname*{i}\mu_{\operatorname*{R}%
}\left\vert b-y\right\vert }$.

\textbf{1st case:} $a_{\operatorname*{R}}=1$ so that $\mu=\mu
_{\operatorname*{R}}=\mu_{\operatorname*{L}}=\sqrt{\lambda}$. Then%
\begin{align*}
g_{a,b,c}\left(  \lambda,y\right)   &  =\frac{\operatorname*{e}%
\nolimits^{\operatorname*{i}\mu\left(  b-y\right)  }-\operatorname*{e}%
\nolimits^{\operatorname*{i}\mu{y}}\cos\mu{b}}{2\operatorname*{i}\mu},\\
f_{a,b,c}\left(  \lambda,y\right)   &  =\frac{\operatorname*{e}%
^{\operatorname*{i}\mu\left(  b-y\right)  }}{2}+\frac{\operatorname*{e}%
^{\operatorname*{i}\mu{y}}}{2\operatorname*{i}}\sin\mu{b}.
\end{align*}
These functions are linearly independent provided $\left(  \alpha
,\beta\right)  =\left(  0,0\right)  $ is the only solution of%
\begin{equation}
\left(  \alpha+\operatorname*{i}\beta\right)  \operatorname*{e}%
\nolimits^{\operatorname*{i}\mu\left(  b-y\right)  }+\left(  \beta\sin\mu
{b}-\alpha\cos\mu{b}\right)  \operatorname*{e}\nolimits^{\operatorname*{i}%
\mu{y}}=0. \label{alphabetacond}%
\end{equation}
Since $\mu\neq0$ the functions $\operatorname*{e}\nolimits^{\operatorname*{i}%
\mu\left(  b-y\right)  }$ and $\operatorname*{e}\nolimits^{\operatorname*{i}%
\mu{y}}$ are linearly independent so that (\ref{alphabetacond}) implies
$\alpha=-\operatorname*{i}\beta$ and $\beta\sin\mu{b}=\alpha\cos\mu{b}$. It is
a simple exercise to verify that $\left(  \alpha,\beta\right)  =\left(
0,0\right)  $ is the only solution so that we proved the lemma for the first case.

\textbf{2nd case: }$a_{\operatorname*{R}}\neq1$ so that $\mu
_{\operatorname*{R}}\neq\mu_{\operatorname*{L}}$. Observe that the function
$g_{a,b,c}\left(  \lambda,\cdot\right)  $ is continuous while $f_{a,b,c}%
\left(  \lambda,\cdot\right)  $ is discontinuous since $\delta\left(
b\right)  =1-a_{\operatorname*{R}}\neq0$. Hence, they are linearly independent
provided $g_{a,b,c}$ is not the zero function. Let $0\leq y\leq b$ so that%
\[
g_{a,b,c}\left(  \lambda,y\right)  =\frac{\operatorname*{e}^{\operatorname*{i}%
\mu_{\operatorname*{L}}b}\operatorname*{e}^{-\operatorname*{i}\mu
_{\operatorname*{L}}y}-\cos\mu_{\operatorname*{L}}b\operatorname*{e}%
^{\operatorname*{i}\mu_{\operatorname*{L}}y}}{2\operatorname*{i}%
\mu_{\operatorname*{L}}}.
\]
Since $\operatorname*{e}^{-\operatorname*{i}\mu_{\operatorname*{L}}y}$ and
$\operatorname*{e}^{\operatorname*{i}\mu_{\operatorname*{L}}y}$ are linearly
independent the function $g_{a,b,c}\left(  \lambda,\cdot\right)  $ is the zero
function if and only if $\operatorname*{e}^{\operatorname*{i}\mu
_{\operatorname*{L}}b}$ and $\cos\mu_{\operatorname*{L}}b$ are zero. However,
this is not possible and we proved the lemma also for the second case.
\end{proof}

If $\lambda$ is not an eigenvalue of (\ref{defl}) the system (\ref{defG1G2})
has a unique solution and the Green's function is well defined. If $\lambda$
approaches an eigenvalue $\lambda_{j}$ the Green's function has a singularity
which is related only to the part $G_{a,b,c}^{\hom}\left(  \lambda,\cdot
,\cdot\right)  $ since $G_{a,b,c}^{\operatorname*{free}}$ is a bounded
function with respect to $\lambda$. The order of singularity (as
$\lambda\rightarrow\lambda_{j}$) depends on the order of the zero of
$\det\mathbf{M}_{a,b,c}\left(  \lambda\right)  $ at $\lambda=\lambda_{j}$.
Expansion of $\det\mathbf{M}_{a,b,c}\left(  \lambda\right)  $ about some
$\lambda_{j}$ leads to%
\begin{equation}
\det\mathbf{M}_{a,b,c}\left(  \lambda\right)  =\sum_{\ell=0}^{\infty}%
\gamma_{\ell}\left(  \lambda-\lambda_{j}\right)  ^{\ell}\text{\quad for some
}\gamma_{\ell}=\gamma_{\ell}\left(  a,b,c,\lambda_{j}\right)  .
\label{detexpansion}%
\end{equation}

\begin{theorem}
Let $\operatorname{Re}a_{\operatorname*{R}}>0$ and let $\left(  u_{j}%
,\lambda_{j}\right)  \in V\backslash\left\{  0\right\}  \times\mathbb{C}%
\backslash\left\{  0\right\}  $ denote an eigenpair of (\ref{defl}) and let
the determinant of $\mathbf{M}_{a,b,c}\left(  \lambda\right)  $ be expanded
according to (\ref{detexpansion}). Let $\nu$ denote the largest integer such
that $\gamma_{\ell}=0$ for $0\leq\ell\leq\nu$ and we assume $\nu<\infty$,
i.e., $\det\mathbf{M}_{a,b,c}$ is not the zero function to avoid pathological
cases. Then, the ascent $\alpha_{j}$ of $\lambda_{j}$ equals $\nu$. The
dimension of the generalized eigenspace is $m_{\operatorname*{alg}}=\nu$ and
spanned by%
\[
w_{\ell}:=\left.  \frac{d^{\ell-1}u_{j}}{d\lambda^{\ell-1}}\right\vert
_{\lambda=\lambda_{j}}\quad\text{for }1\leq\ell\leq\nu.
\]
The functions $w_{\ell}$ belong to $\mathcal{N}\left(  \left(  \mathcal{A}%
-\lambda_{j}I\right)  ^{\kappa}\right)  $ for $\kappa=\ell$ but not for
$\kappa<\ell$.
\end{theorem}

\begin{proof}
First, we determine the order of singularity at $\lambda=\lambda_{j}$ of the
Green's function. It suffices to study the part $G_{a,b,c}^{\hom}\left(
\lambda,\cdot,\cdot\right)  $ since $G_{a,b,c}^{\operatorname*{free}}\left(
\lambda,\cdot,\cdot\right)  $ does not introduce poles. The singularity at
$\lambda_{j}$ is induced via the coefficients $G_{1}$, $G_{2}$ as the solution
of (\ref{defG1G2}). For $\lambda$ being not an eigenvalue of (\ref{defl}) we
have\footnote{For a $2\times2$ matrix $\mathbf{A}=\left[
\begin{array}
[c]{ll}%
a & b\\
c & d
\end{array}
\right]  $ we set $\mathbf{A}^{\times}\mathbf{:}=\left[
\begin{array}
[c]{ll}%
d & -b\\
-c & a
\end{array}
\right]  $. For $\mathbf{u}=\left(  u_{i}\right)  _{i=1}^{2}$, $\mathbf{v}%
=\left(  v_{i}\right)  _{i=1}^{2}\in\mathbb{C}^{2}$, we denote by
$\left\langle \cdot,\cdot\right\rangle $ the \textit{bilinear} form
$\left\langle \mathbf{u},\mathbf{v}\right\rangle =u_{1}v_{1}+u_{2}v_{2}$.}%
\begin{equation}
G_{a,b,c}^{\hom}\left(  \lambda,x,y\right)  :=\left\{
\begin{array}
[c]{ll}%
\frac{\left\langle \mathbf{M}_{a,b,c}^{\times}\left(  \lambda\right)  \left(
\begin{array}
[c]{c}%
g_{a,b,c}\left(  \lambda,y\right)  \\
f_{a,b,c}\left(  \lambda,y\right)
\end{array}
\right)  ,\left(
\begin{array}
[c]{c}%
v_{a,b,c}^{\operatorname*{L}}\left(  \lambda,x\right)  \\
0
\end{array}
\right)  \right\rangle }{\det\mathbf{M}_{a,b,c}\left(  \lambda\right)  } &
\text{in }\Omega_{1}\times\Omega,\\
& \\
\frac{\left\langle \mathbf{M}_{a,b,c}^{\times}\left(  \lambda\right)  \left(
\begin{array}
[c]{c}%
g_{a,b,c}\left(  \lambda,y\right)  \\
f_{a,b,c}\left(  \lambda,y\right)
\end{array}
\right)  ,\left(
\begin{array}
[c]{c}%
0\\
v_{a,b,c}^{\operatorname*{R}}\left(  \lambda,x\right)
\end{array}
\right)  \right\rangle }{\det\mathbf{M}_{a,b,c}\left(  \lambda\right)  } &
\text{in }\Omega_{2}\times\Omega.
\end{array}
\right.  \label{Grep}%
\end{equation}
Poles are introduced to the Green's function via the zeroes of $\det
\mathbf{M}_{a,b,c}\left(  \lambda\right)  $ (cf. (\ref{detexpansion})); to
determine their orders we also have to investigate whether the numerators in
(\ref{Grep}) can be the zero function for certain values of $\lambda_{j}$. In
the following we will prove (by contradiction) that the two brackets
$\left\langle \cdot,\cdot\right\rangle $ in (\ref{Grep}) are \textit{not} the
zero function in a neighborhood of an eigenvalue. Since $v_{a,b,c}%
^{\operatorname*{L}}\left(  \lambda,\cdot\right)  $ and $v_{a,b,c}%
^{\operatorname*{R}}\left(  \lambda,\cdot\right)  $ are not the zero function
in a neighborhood of an eigenvalue $\lambda_{j}$ (cf. Lemma \ref{Lemlambdaj})
we conclude that $\mathbf{M}_{a,b,c}^{\times}\left(  \lambda\right)  \left(
\begin{array}
[c]{c}%
g_{a,b,c}\left(  \lambda,y\right)  \\
f_{a,b,c}\left(  \lambda,y\right)
\end{array}
\right)  $ must be the zero function. Recall that $\lambda\neq0$. From Lemma
\ref{LemLinIndep} we know that $g_{a,b,c}\left(  \lambda,\cdot\right)  $ and
$f_{a,b,c}\left(  \lambda,\cdot\right)  $ are linearly independent so that
there exists two values $y_{1},y_{2}\in\Omega$ such that the vectors $\left(
g_{a,b,c}\left(  \lambda,y_{i}\right)  ,f_{a,b,c}\left(  \lambda,y_{i}\right)
\right)  ^{\intercal}$, $i=1,2$, are linearly independent. Hence
$\mathbf{M}_{a,b,c}\left(  \lambda\right)  $ must be the zero matrix in order
that the two brackets $\left\langle \cdot,\cdot\right\rangle $ in (\ref{Grep})
could be the zero function. However, for $\lambda\neq0$ the matrix entries
$v_{a,b,c}^{\operatorname*{L}}\left(  \lambda,b\right)  =\sin\mu
_{\operatorname*{L}}b$ and $\partial_{x}v_{a,b,c}^{\operatorname*{L}}\left(
\lambda,b\right)  =\mu_{\operatorname*{L}}\cos\mu_{\operatorname*{L}}b$ cannot
be zero simultaneously and, hence, the matrix $\mathbf{M}_{a,b,c}\left(
\lambda\right)  $ cannot be the zero matrix (cf. Rem. \ref{RemZeroMat}). From
(\ref{detexpansion}) and $\gamma_{\ell}=0$ for $0\leq\ell\leq\nu$ we now can
conclude that the order of the pole of the Green's function at $\lambda_{j}$
equals $\nu$.

From \cite[Thm. 3.1]{M} we know that the ascent of $\lambda_{j}$ equals the
order of the pole $\nu$.

Since the geometric multiplicity of eigenvalues for problem (\ref{defl})
equals $1$ (cf. Rem. \ref{RemZeroMat}) we get by induction that the space of
generalized eigenfunctions are spanned by the solutions of the following
sequence of problems: Set $H_{0}:=\left\{  0\right\}  $. For $1\leq\ell\leq
\nu$, let $w_{\ell}\in V\backslash H_{\ell-1}$ be a solution of%
\[
\left(  \mathcal{A}-\lambda_{j}I\right)  ^{\ell}w_{\ell}=0\quad\text{and
set\quad}H_{\ell}:=H_{\ell-1}+\operatorname*{span}\left\{  w_{\ell}\right\}
.
\]
Clearly, we have $w_{1}=u_{j}$ and $H_{1}$ is the (one-dimensional) eigenspace
of $\lambda_{j}$. For $\ell=2,\ldots,\nu$ we obtain by induction%
\[
\left(  \mathcal{A}-\lambda_{j}I\right)  ^{\ell}w_{\ell-1}=\left(
\mathcal{A}-\lambda_{j}I\right)  \left(  \left(  \mathcal{A}-\lambda
_{j}I\right)  ^{\ell-1}w_{\ell-1}\right)  =0.
\]
We differentiate this equation with respect to $\lambda_{j}$ and obtain%
\[
-\ell\left(  \mathcal{A}-\lambda_{j}I\right)  ^{\ell-1}w_{\ell-1}+\left(
\mathcal{A}-\lambda_{j}I\right)  ^{\ell}\partial_{\lambda_{j}}w_{\ell-1}=0.
\]
The first summand vanishes by induction so that $w_{\ell}=\partial
_{\lambda_{j}}w_{\ell-1}$ is in $H_{\ell}$ if we prove that it is not the zero
function. Since $u_{j}$ is an eigenfunction of (\ref{defl}) it is the non-zero
function: on $\Omega_{1}$ it is a multiple of $\sin\sqrt{\lambda_{j}}x$ and on
$\Omega_{2}$ a linear combination of $\cos\sqrt{\lambda_{j}%
/a_{\operatorname*{R}}}x$ and $\sin\sqrt{\lambda_{j}/a_{\operatorname*{R}}}x$
so that no derivative with respect to $\lambda$ is the zero function. Hence
$w_{\ell}\in H_{\ell}$.
\end{proof}

\begin{remark}
The functions $\gamma_{\ell}\left(  a,b,c,\lambda\right)  $ in
(\ref{detexpansion}) are transcendental complex\--val\-ued functions and it is a
non-trivial task to determine coefficients $a,b,c,\lambda$ such that
$\gamma_{\ell}$ is zero for $\ell=0,\ldots,\nu$ for some $\nu>0$. In
\cite{GasserMaster} a procedure is described how such parameter configurations
can be computed to high precision. In the setting of our paper, we were able
to choose these parameters,
using the nonlinear solver of Mathematica applied to
the symbolic expression of $\gamma_{\ell}\left(  a,b,c,\lambda\right)  $
and carefully chosen starting values, such that $\gamma_{\ell}\left(  a,b,c\right)  =0$
for $\ell=0,1,2,3$. We conjecture that it is not possible to find
configurations for problem (\ref{defl}) such that $\gamma_{\ell}$ vanishes at
a higher order.
\end{remark}

\section{Numerical Experiments\label{SecNumExp}}

In this section we present several numerical experiments that indicate that
the Babu{\v{s}}ka-Osborn theory is sharp for defective eigenvalues. The
numerical experiments below verify that the eigenvalue errors of a defective
eigenvalue have reduced convergence rates while the mean eigenvalue error
converges with the full rate. We construct two main examples based on the
construction in Section \ref{SecGreensFct}, one has full regularity and one
has reduced regularity.

In the following figures, we display the errors
in terms of the number of degrees of freedom $N$, where $h\approx N^{-1/d}$
for uniform meshes.

\subsection{Regular Example}

\begin{figure}[ptb]
\includegraphics[width=0.49\textwidth]{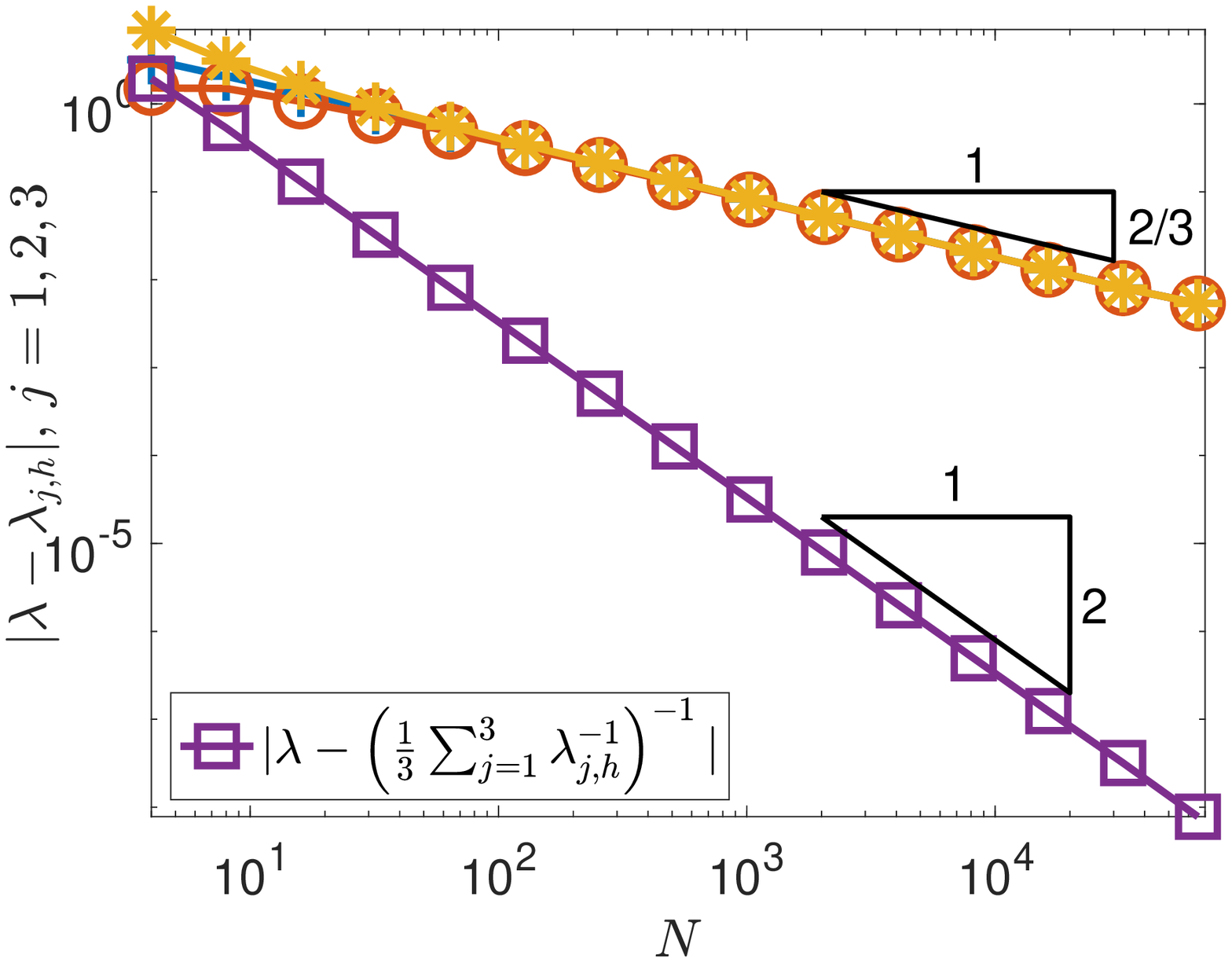}
\includegraphics[width=0.49\textwidth]{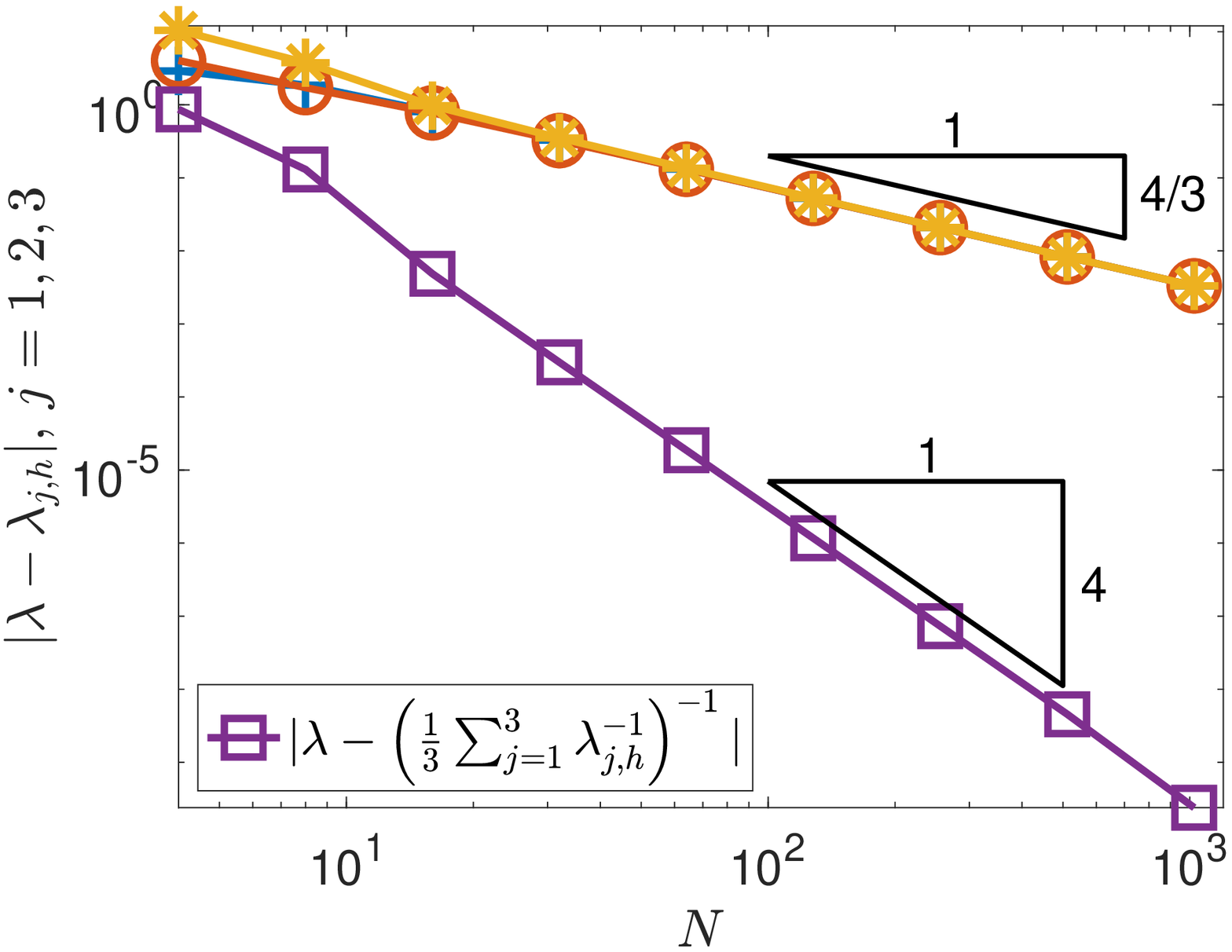}\caption{Regular
example with $\mathbb{P}_{1}$, and $\mathbb{P}_{2}$ FEM in 1d.}%
\label{fig:SmoothConvergence1d}%
\end{figure}
\begin{figure}[ptb]
\centering\includegraphics[width=0.49\textwidth]{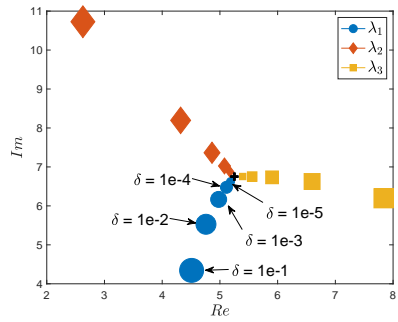}\caption{Sensitivity
of eigenvalues due to $\delta$ perturbations of the real part of $c$.}%
\label{fig:SmoothSensitivity1d}%
\end{figure}
\begin{figure}[ptb]
\centering\includegraphics[width=0.49\textwidth]{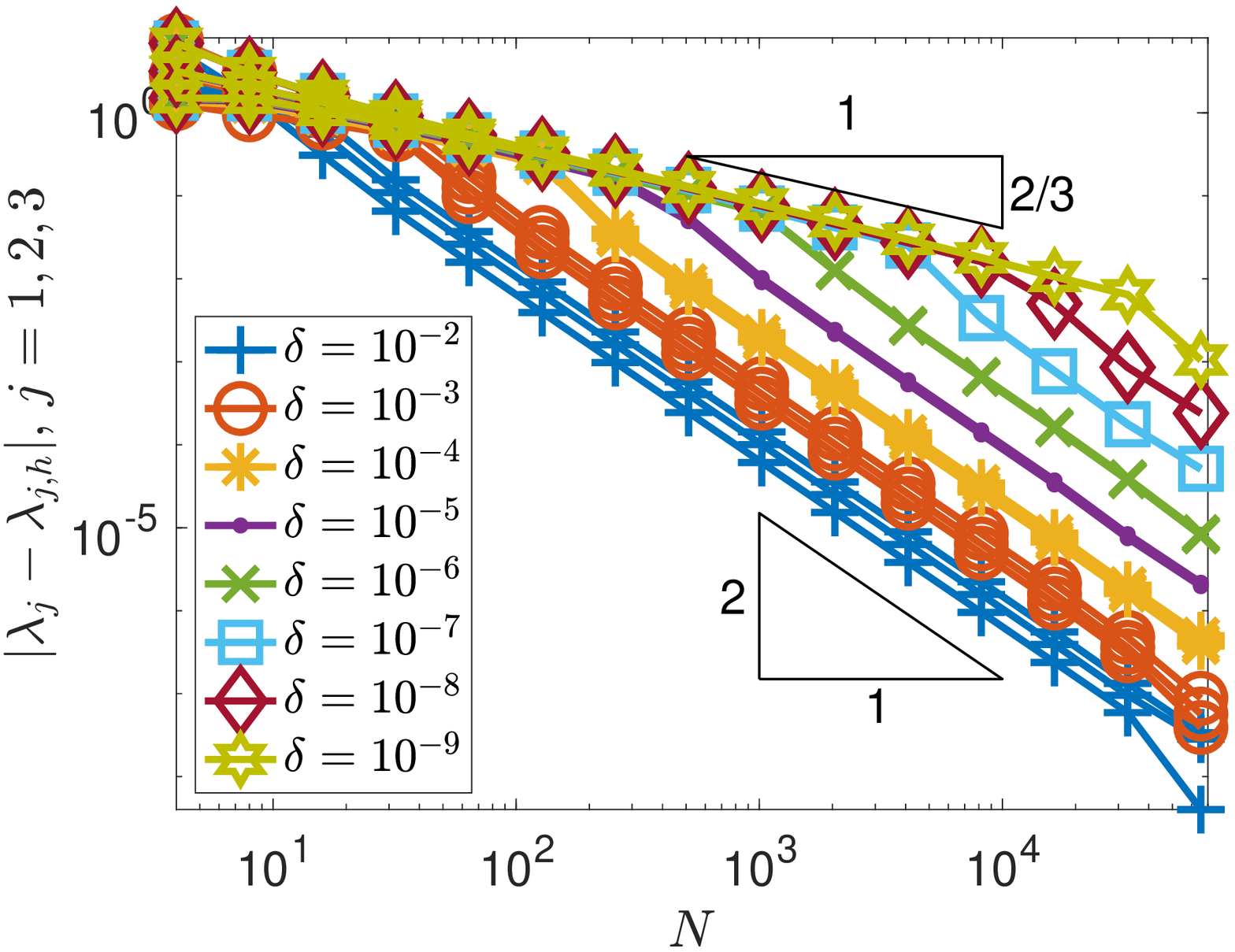}
\includegraphics[width=0.49\textwidth]{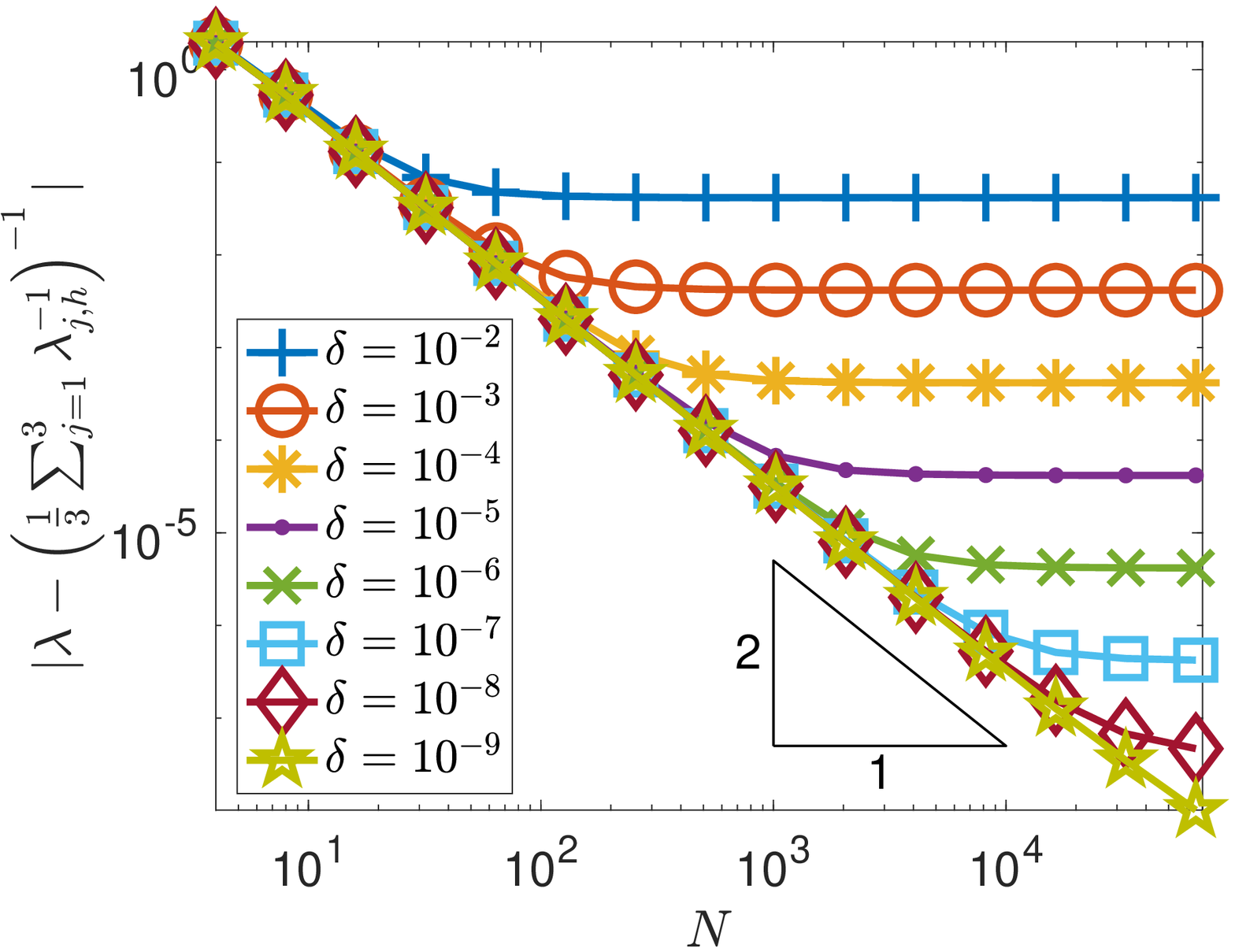}\caption{Convergence
history for $\mathbb{P}_{1}$ FEM in 1d for $\delta$ perturbed problems with
clustered eigenvalues.}%
\label{fig:SmoothSensitivity1d_2}%
\end{figure}

In this example we align the jump of the diffusion coefficient with the mesh,
so that the diffusion coefficient is piecewise constant on refined meshes.
Although the continuous eigenfunctions are not globally smooth due to the
jumping coefficient, they are piecewise smooth and if the mesh contains the
jump point as a mesh point we can expect that the convergence orders are
\textit{not} reduced due to lower global regularity. In this light, we call
this set of examples \textit{regular examples.}

\subsection{One-dimensional Example}

Let $\Omega=(0,1)$ with $a=1$ on $[0,1/2]$ and
$a=a_{\operatorname*{R}}$ on $(1/2,1]$, using the construction of the previous
section, we compute
\begin{align*}
a_{\operatorname*{R}}  &  =0.1069220800406739+0.08937533852238478i,\\
c  &  =-0.9634059612381408+0.5989684988897067i,
\end{align*}
for the first (smallest in magnitude) complex eigenvalue
\[
\lambda=5.250721274740938+6.750931815875402i,
\]
which has algebraic multiplicity $m_{alg}=3$ and ascent $\alpha=3$ by construction.

In Figure~\ref{fig:SmoothConvergence1d} we observe convergence rates according
to the theory, in case of the $\mathbb{P}_{1}$ finite element method the
convergence is of order $\mathcal{O}(N^{-2/3})$ (due to $\alpha=3$) for the
eigenvalue errors $|\lambda-\lambda_{j,h}|$, $j=1,2,3$, and optimal
convergence $\mathcal{O}(N^{-2})$ for the mean eigenvalue error. For the
second order $\mathbb{P}_{2}$ finite element method we observe twice the
convergence rate, i.e. $\mathcal{O}(N^{-4/3})$ for the eigenvalue errors, and
$\mathcal{O}(N^{-4})$ for the mean eigenvalue error which show that the
theoretical predicted rates are sharp for these examples.

Next, we investigate the sensitivity of the defective eigenvalue $\lambda$.
Since the defect is very sensitive towards the choice of the parameters
$a_{R}$ and $c$, we perturb only the real part of $c$ by adding a small (real)
value $\delta$. In Figure~\ref{fig:SmoothSensitivity1d}, we observe that even
very small perturbations $\delta$, immediately lead to a splitting of the
defective eigenvalue into three clustered eigenvalues. Even a relatively small
perturbation $\delta=10^{-2}$, of about 1\%, already leads to a significant
separation of the eigenvalues of size greater than $2$.

We investigate the transition of the defective eigenvalue into a separated
cluster of eigenvalues in more detail and make the following observations in
Figure~\ref{fig:SmoothSensitivity1d_2}. In the left figure we display the
eigenvalue errors $|\lambda_{j}-\lambda_{j,h}|$, $j=1,2,3$, for $\mathbb{P}%
_{1}$ finite elements and different perturbations $\delta$ towards precomputed
reference values $\lambda_{j}$ for the clustered eigenvalues. We computed the
reference values with higher order $\mathbb{P}_{3}$ finite elements on fine
meshes with high accuracy. In the right figure, we show the convergence of the
mean eigenvalue error towards the defective eigenvalue $\lambda$. We observe
that for $\delta=10^{-2}$ the eigenvalues are well separated, hence the three
distinct eigenvalues converge with optimal rates and the mean eigenvalue error
does not converge towards the defective eigenvalue $\lambda$. Interestingly,
for smaller values of $\delta$, we observe that there seems to be a resolution
barrier. Before a certain resolution is reached, we observe that the
eigenvalue errors show the reduced convergence rate of approximating a
defective eigenvalue, and even the mean value converges towards the defective
eigenvalue $\lambda$. Once the mesh is fine enough, so that the clustered
eigenvalues can be separated also on the discrete level, the eigenvalue errors
converge with optimal rates and their mean value stops converging towards the
defective eigenvalue.

In \cite{GasserMaster}, other explicit choices of parameters are given such
that the eigenvalue of the elliptic boundary value problem is defective.

\subsection{Higher Dimensions}

\begin{figure}[ptb]
\includegraphics[width=0.49\textwidth]{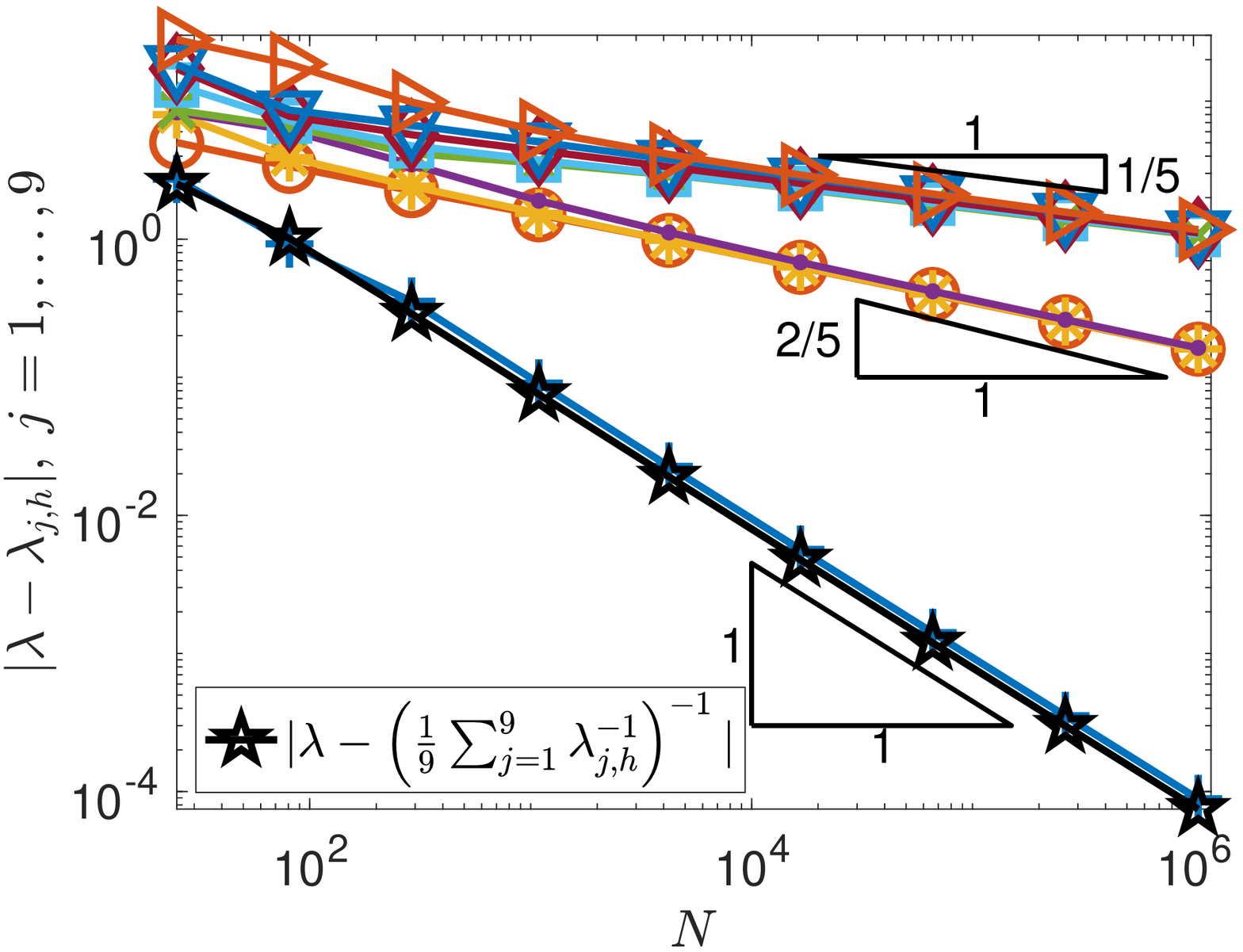}
\includegraphics[width=0.49\textwidth]{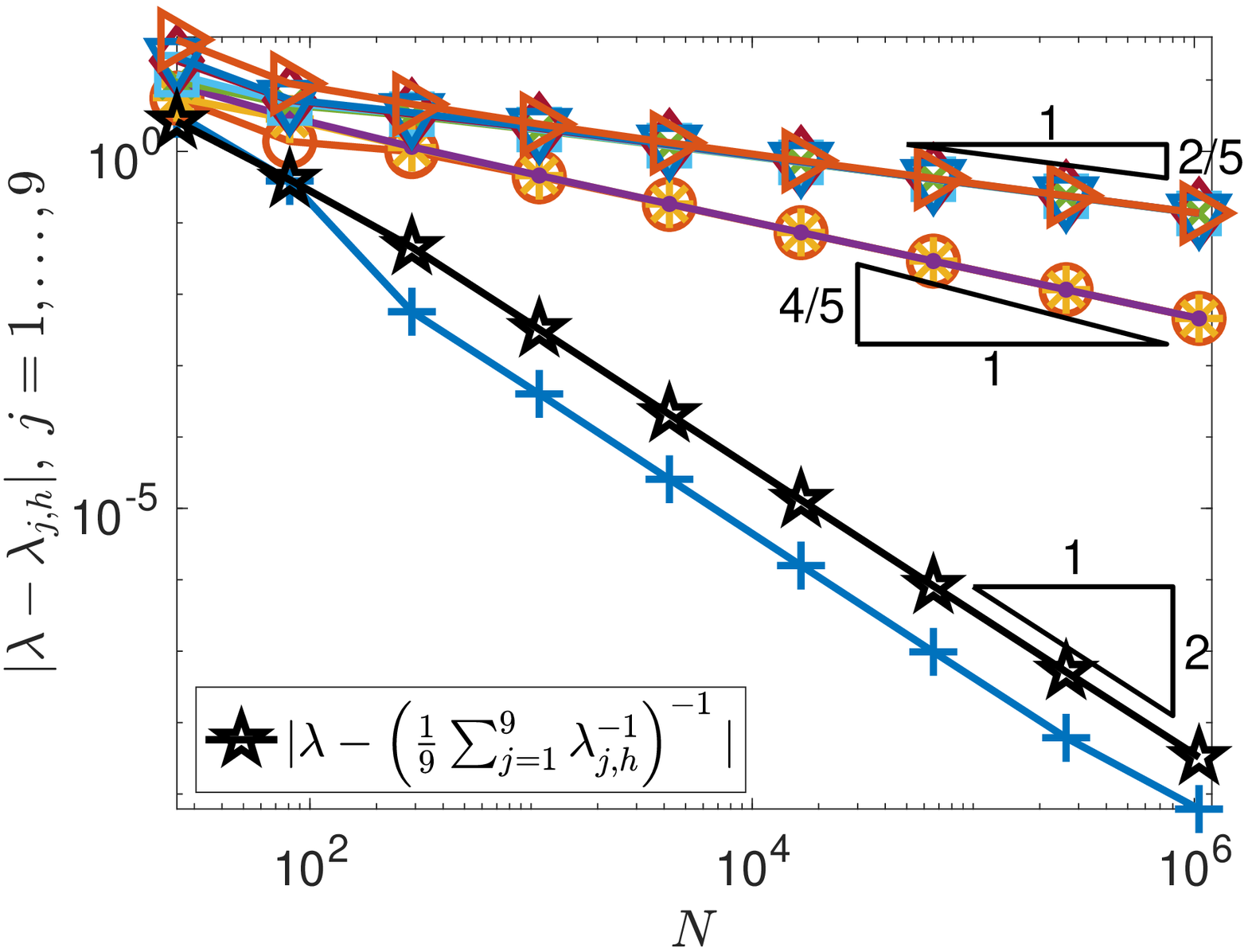}\caption{Regular
example with $\mathbb{P}_{1}$, and $\mathbb{P}_{2}$ FEM in 2d.}%
\label{fig:smooth2d}
\end{figure}
\begin{figure}[ptb]
\includegraphics[width=0.49\textwidth]{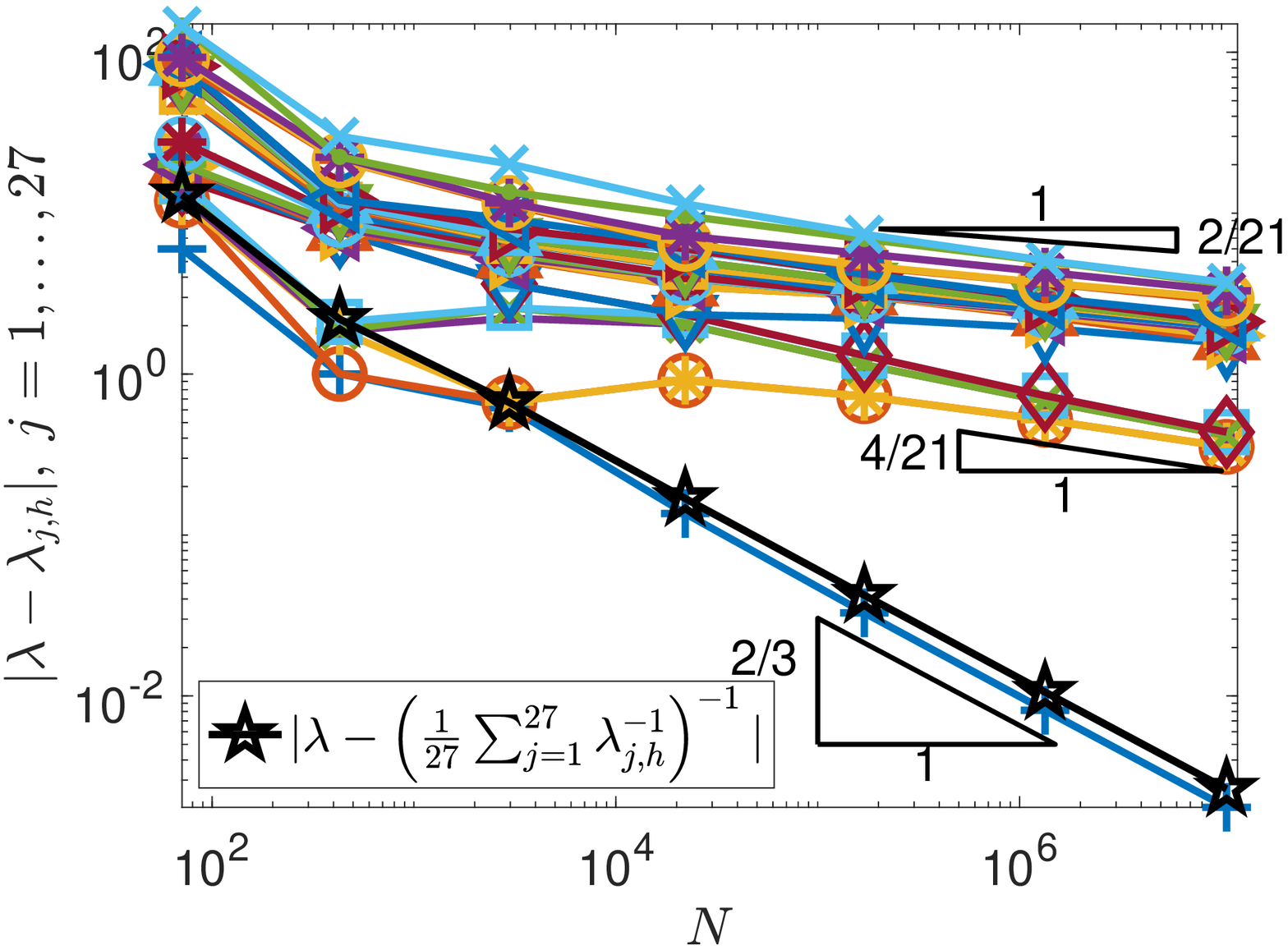}
\includegraphics[width=0.49\textwidth]{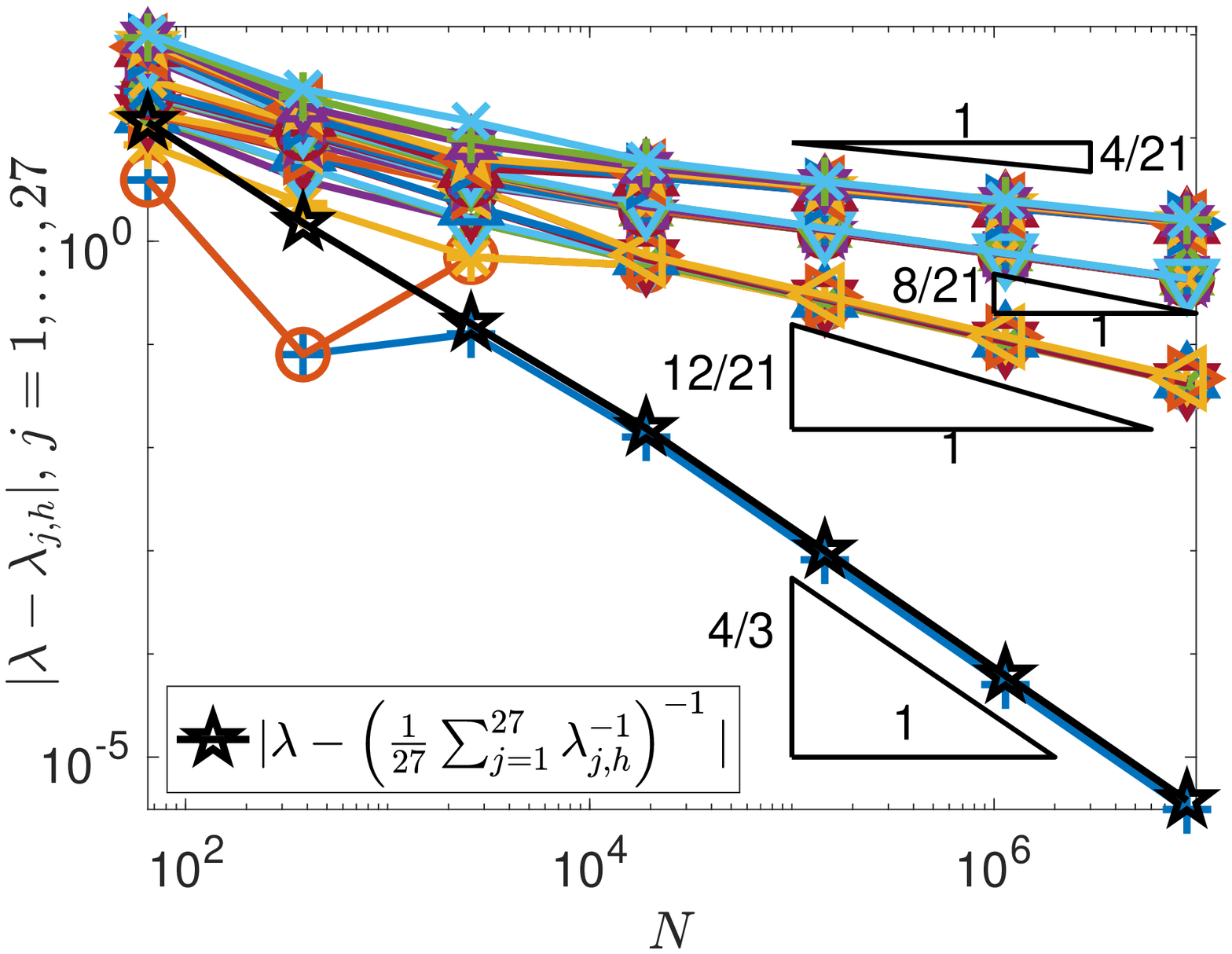}\caption{Regular
example with $\mathbb{P}_{1}$, and $\mathbb{P}_{2}$ FEM in 3d.}%
\label{fig:smooth3d}
\end{figure}

We extend the one-dimensional example to higher dimensions
$d=2,3$ by taking the tensor product of the (generalized) eigenfunctions in
$x$, $y$ and $z$ coordinates, which leads to the eigenvalue $2\lambda$ in two
dimensions of algebraic multiplicity $m_{alg}=9$, and the eigenvalue
$3\lambda$ for $d=3$ with algebraic multiplicity $m_{alg}=27$. The diffusion
coefficient and the boundary conditions are extended by tensorization to
higher dimensions as well.

In two dimensions, we observe in Figure~\ref{fig:smooth2d} for $\mathbb{P}_{1}$ finite elements convergence
of at least $\mathcal{O}(N^{-1/5})$ for the eigenvalue errors and
$\mathcal{O}(N^{-1})$ for the mean eigenvalue error. For $\mathbb{P}_{2}$
finite elements we observe twice the convergence, namely at least
$\mathcal{O}(N^{-2/5})$ for the eigenvalue errors and $\mathcal{O}(N^{-2})$
for the mean eigenvalue error. This shows numerically the ascent $\alpha=5$.
In addition, we observe that some discrete eigenvalues converge with rates in
between those two extreme cases. Some eigenvalues converge with order close to
$\mathcal{O}(N^{-2/5})$ for $\mathbb{P}_{1}$ finite elements and close to
$\mathcal{O}(N^{-4/5})$ for $\mathbb{P}_{2}$ finite elements. Note that one
eigenvalue seems to correspond to a discrete eigenvector, hence converges with
the optimal rate.

For $d=3$, we observe in Figure~\ref{fig:smooth3d} convergence rates of the eigenvalue errors as low as
$\mathcal{O}(N^{-2/21})$ for $\mathbb{P}_{1}$ finite elements and
$\mathcal{O}(N^{-4/21})$ for $\mathbb{P}_{2}$ finite elements, which indicate
the ascent $\alpha=7$. Again the mean eigenvalue errors converge optimally.
Note that some discrete eigenvalues converge with rates in between the optimal
and reduced ones, and that one eigenvalue converges with optimal rate. In
particular for $\mathbb{P}_{2}$ finite elements we observe that some
eigenvalues converge with order $\mathcal{O}(N^{-8/21})$, and some even with
order $\mathcal{O}(N^{-12/21})$, which relates to convergence order of
$\mathcal{O}(h^{2k/7})$, for $k=2,3$. This confirms the impressive sharpness
of the theory and that in principle any convergence order $\mathcal{O}%
(h^{pk/\alpha})$, for $k=1,\ldots,\alpha$, can occur, not only in theory, but
as we have demonstrated also in practical computations.

\subsection{Examples with Reduced Regularity}

\begin{figure}[ptb]
\includegraphics[width=0.49\textwidth]{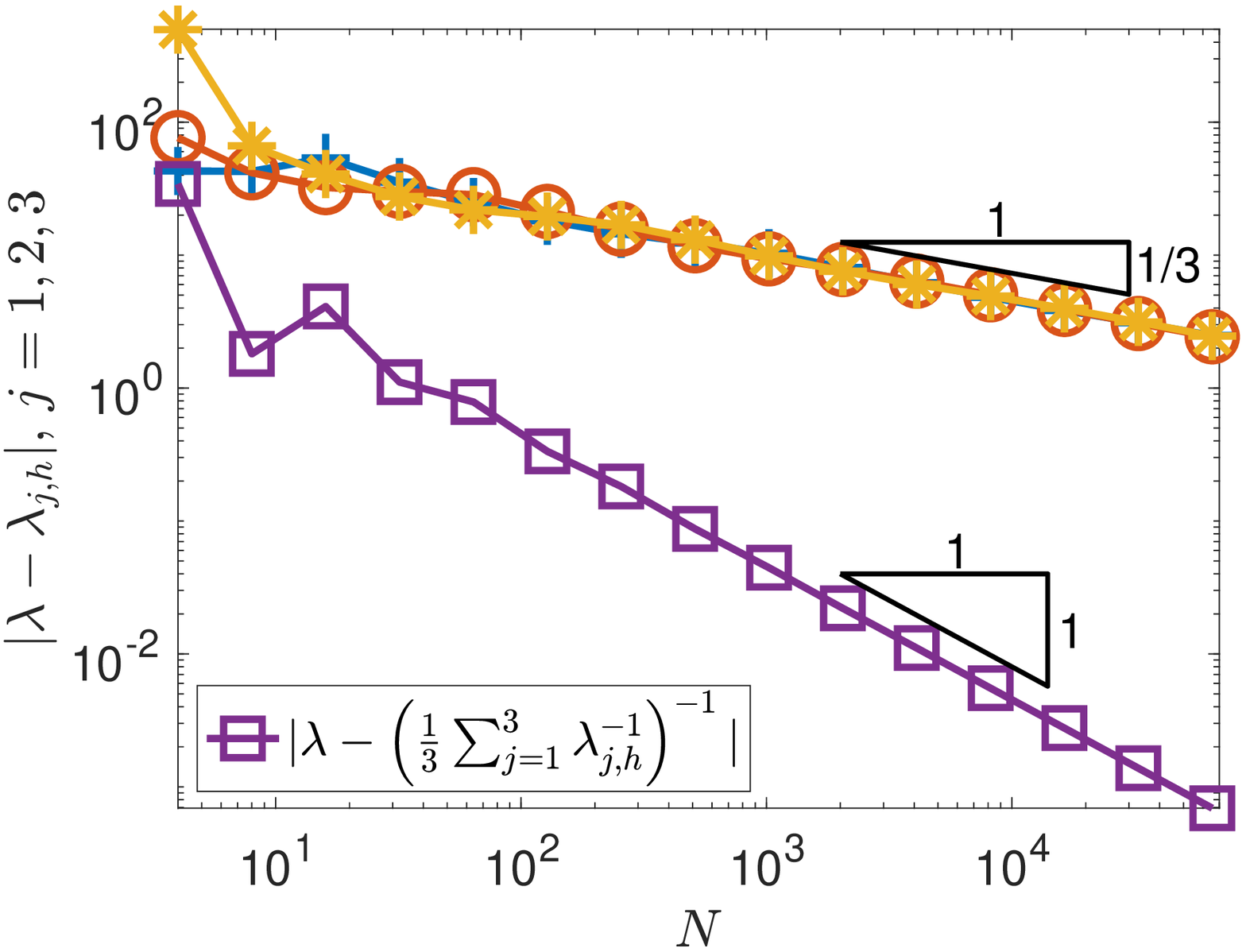}
\includegraphics[width=0.49\textwidth]{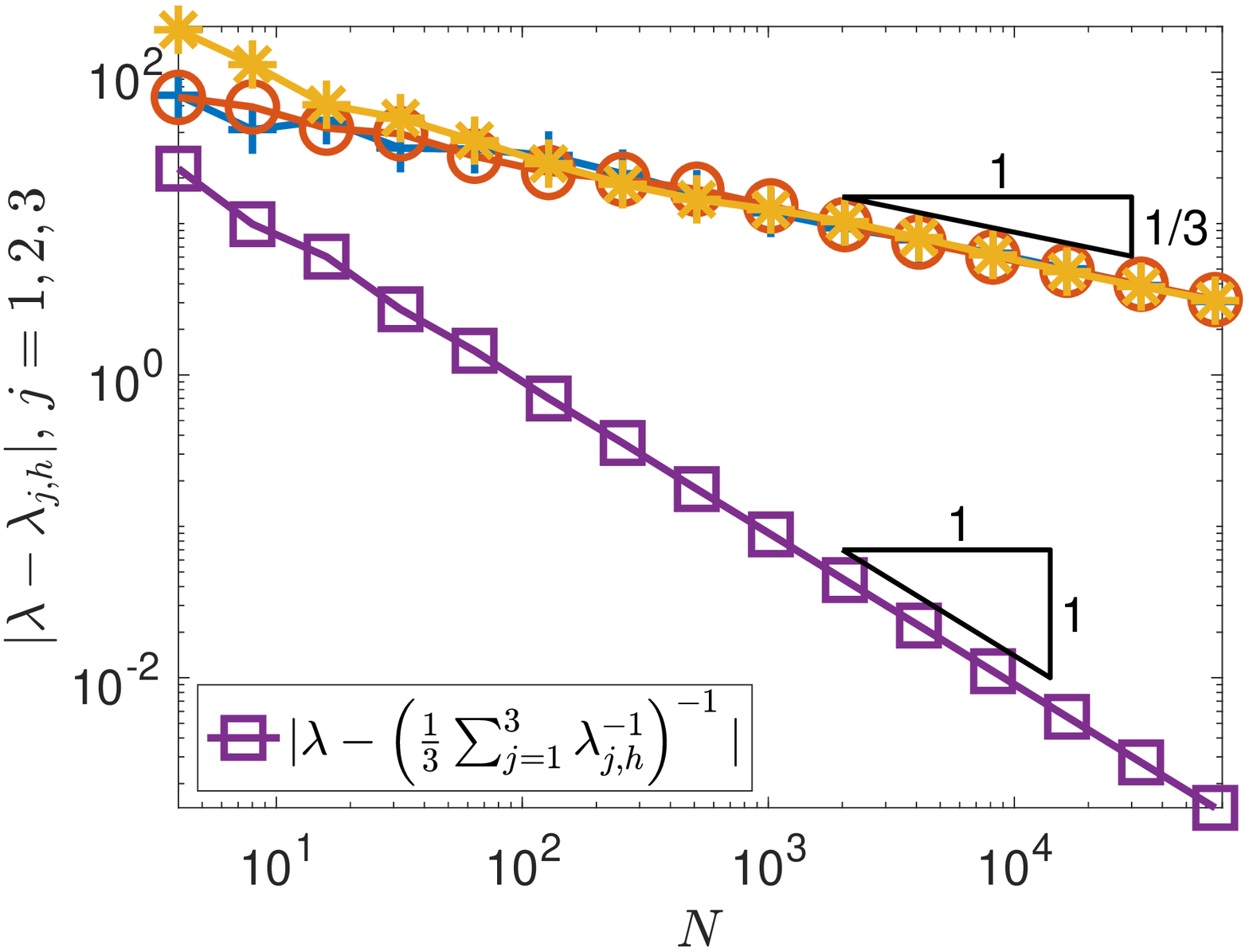}\caption{Example
with reduced regularity for $\mathbb{P}_{1}$, and $\mathbb{P}_{2}$ FEM in 1d.}%
\label{fig:NonSmooth1d}%
\end{figure}
\begin{figure}[ptb]
\centering\includegraphics[width=0.49\textwidth]{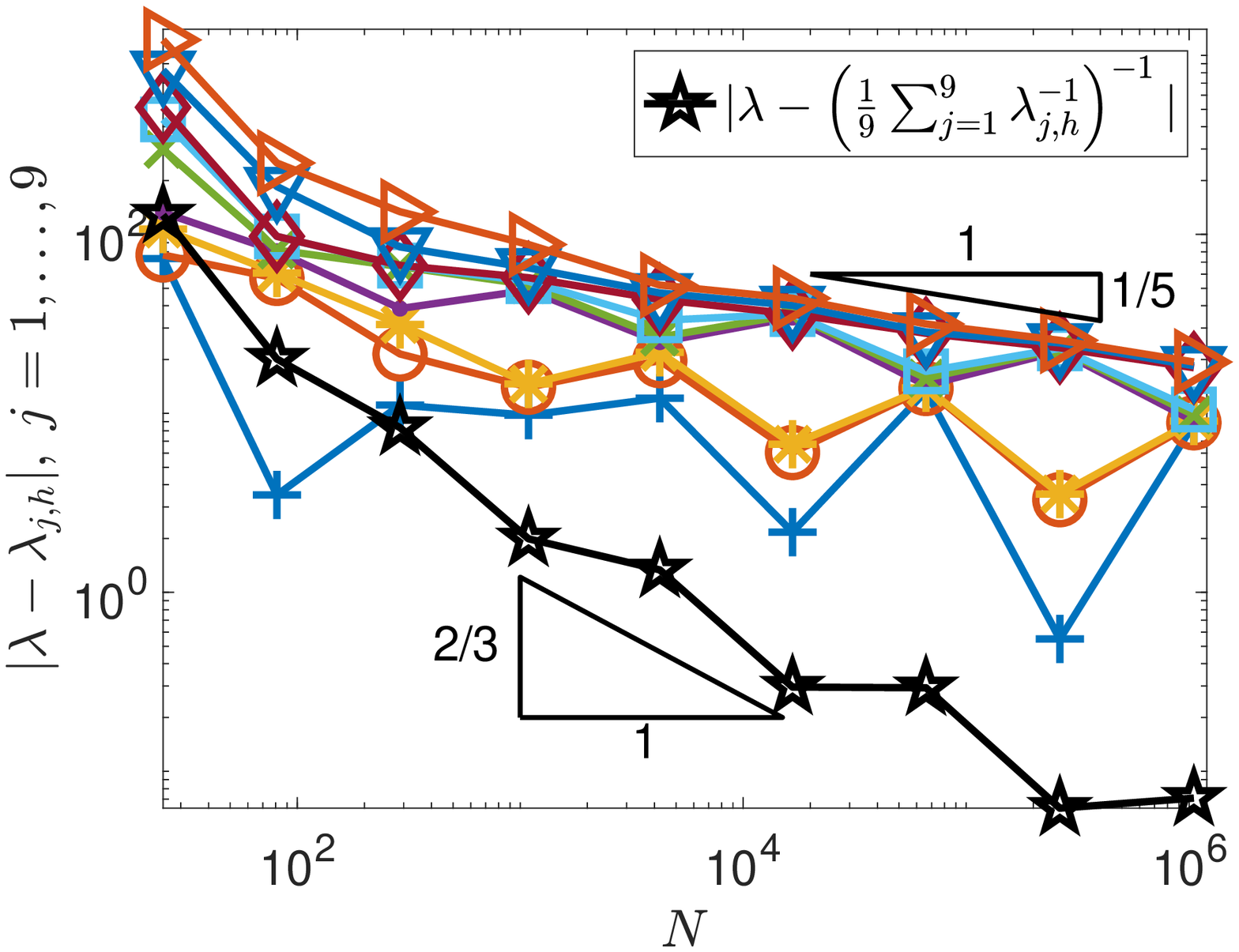}
\includegraphics[width=0.49\textwidth]{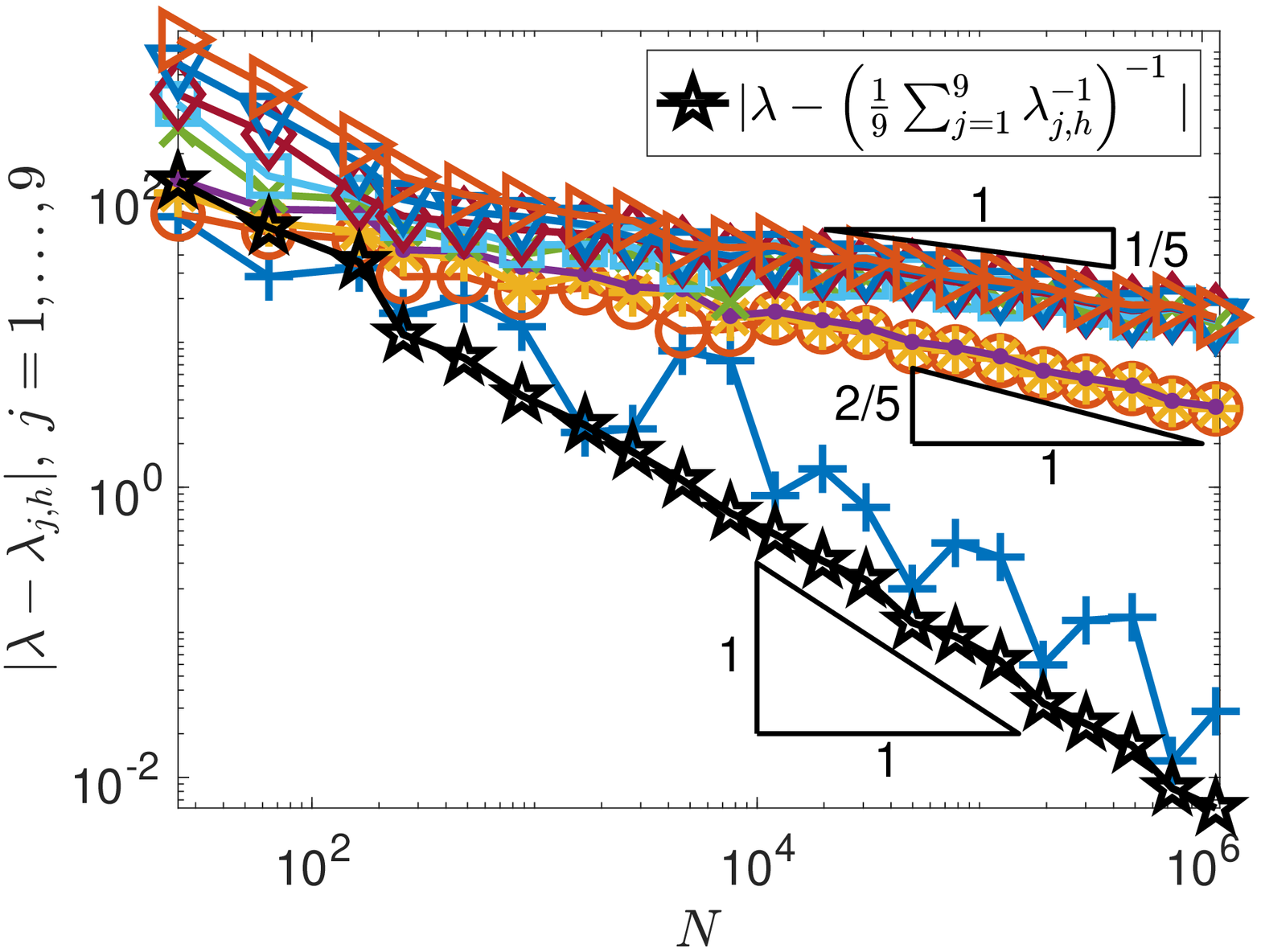}\caption{Example
with reduced regularity for $\mathbb{P}_{1}$ FEM on uniform and adaptive
meshes in 2d.}%
\label{fig:adaptivity}%
\end{figure}

Here, we choose the coefficient $a$ such that its jump is not
aligned with any (refined) mesh. Therefore, we consider $\Omega=(0,1)$ with
$a=1$ on $[0,1/3]$ and $a=a_{\operatorname*{R}}$ on $(1/3,1]$, with
\begin{align*}
a_{\operatorname*{R}}  &  =8.834634001449438+2.381273183203226i,\\
c  &  =-23.62602259938114+23.10185194698031i,
\end{align*}
and the first (smallest in magnitude) complex eigenvalue
\[
\lambda=72.26224904068889+65.85698689932984i.
\]
By construction $\lambda$ has algebraic multiplicity $m_{alg}=3$ and ascent
$\alpha=3$ for $d=1$. As in the previous example, the tensor product of the
(generalized) eigenfunctions leads to the eigenvalue $2\lambda$ with algebraic
multiplicity $m_{alg}=9$ in two dimensions, and numerically we observe the
ascent $\alpha=5$.

Note that since the mesh is not aligned with the jump of the diffusion
coefficient, the regularity of the (generalized) eigenfunctions are reduced to
$H^{1+r}\left(  \Omega\right)  $ for any $0<r<1/2$. Therefore, we observe
reduced convergence of the eigenvalues due to the reduced convergence of the
(generalized) eigenfunctions on uniform meshes.

In Figure~\ref{fig:NonSmooth1d}, we see that the convergence is reduced by two
separate issues: the reduced regularity and the large defect of the
eigenvalue. We observe the theoretically expected suboptimal convergence rates
of the mean eigenvalue error of $\mathcal{O}(N^{-1})$ for both $\mathbb{P}%
_{1}$ and $\mathbb{P}_{2}$ finite elements due to the reduced regularity. The
convergence of the eigenvalue errors is even further reduced due to the defect
$\alpha=3$, hence the convergence is only of order $\mathcal{O}(N^{-1/3})$.

The situation in two dimensions is less clear from the numerical point of
view. In Figure~\ref{fig:adaptivity}, we observe in the left figure reduced
rates of the mean eigenvalue error for $\mathbb{P}_{1}$ finite elements on
uniform meshes, but still at worst $\mathcal{O}(N^{-1/5})$ convergence of the
eigenvalues, which is expected from the defect of $\lambda$, but is not
further decreased by the low regularity. This might be a pre-asymptotic
effect. In the right figure we use an adaptive mesh refinement algorithm
\cite{GC,GGMO,HR,YSBL}. Based on the previous observation, that even very
small perturbations lead to a split of the defective eigenvalue into clustered
eigenvalues, we measure the error of the defective eigenvalue, as if it was a
cluster of eigenvalues, with the a posteriori error estimator
\begin{align*}
\eta_{h}^{2}  &  :=\sum_{j=1}^{m}\sum_{T\in\mathcal{T}}\big(h_{T}^{2}%
\Vert\Delta u_{j,h}+\lambda_{j,h}u_{j,h}\Vert_{0,T}^{2}+\sum_{E\subset\partial
T\backslash\partial\Omega}h_{E}\Vert\lbrack a\nabla u_{j,h}\cdot
\mathbf{n}]\Vert_{0,E}^{2}\\
&  \quad+\sum_{E\subset\partial T\cap\Gamma_{\operatorname*{R}}}h_{E}\Vert a\nabla
u_{j,h}\cdot\mathbf{n}+cu_{j,h}\Vert_{0,E}^{2}\\
&  \quad+ h_{T}^{2}\Vert\Delta u_{j,h}^{*}+\lambda_{j,h}^{*}u_{j,h}^{*}%
\Vert_{0,T}^{2}+\sum_{E\subset\partial T\backslash\partial\Omega}h_{E}%
\Vert\lbrack\overline{a}\nabla u_{j,h}^{*}\cdot\mathbf{n}]\Vert_{0,E}^{2}\\
&  \quad+\sum_{E\subset\partial T\cap\Gamma_{\operatorname*{R}}}h_{E}\Vert\overline
{a}\nabla u_{j,h}^{*}\cdot\mathbf{n}+\overline{c}u_{j,h}^{*}\Vert_{0,E}%
^{2}\big),
\end{align*}
where $(\lambda_{j,h}^{*},u_{j,h}^{*})$ denotes the $j$-th eigenpair of the
adjoint eigenvalue problem. 
Despite that $\eta_{h}^{2}$ provides only a valid upper
bound for clustered eigenvalues, we observe in Figure~\ref{fig:adaptivity}
that adaptive mesh-refinement based on $\eta_h^2$ leads to
optimal convergence of the mean eigenvalue error.
By construction, $\eta_h^2$ cannot give any a
posteriori information about the ascent of the eigenvalue.
Nevertheless, this experiment illustrates that an error estimator is in principle able to heuristically 
detect defective eigenvalues from their reduced convergence rates,
although a theoretical foundation has still to be developed.

\section{Conclusions}

We described a constructive way of deriving benchmark problems with highly
defective eigenvalues. We provided the parameters for two such examples. We
confirmed in numerical experiments that the Babu{\v{s}}ka-Osborn theory is
sharp and that convergence rates between the two extreme cases do occur in
practical computations. Since even for non-smooth eigenfunctions, the mean
eigenvalue error converges faster, one is in principle able to detect
defective eigenvalues numerically by tracking the convergence behavior of the
eigenvalues and the mean eigenvalue error on uniformly or adaptively refined
meshes. If the mean eigenvalue error converges faster, that means that the
eigenvalue is defective.

\end{document}